\documentclass[aihp]{imsart}

\RequirePackage{amsthm,amsmath,amsfonts,amssymb}
\RequirePackage[numbers,sort&compress]{natbib}
\RequirePackage[colorlinks,citecolor=blue,urlcolor=blue]{hyperref}
\RequirePackage{graphicx}
\usepackage{indentfirst}
\usepackage{caption}

\startlocaldefs
\theoremstyle{plain}

\newtheorem{theorem}{Theorem}[section]
\newtheorem*{theorem 2.12}{Theorem 2.12}
\newtheorem{lemma}[theorem]{Lemma}
\newtheorem{definition}[theorem]{Definition}
\newtheorem{example}[theorem]{Example}
\newtheorem{corollary}[theorem]{Corollary}
\theoremstyle{remark}
\newtheorem{remark}{Remark}

\endlocaldefs

\begin{document}
\begin{frontmatter}
\title{Onsager-Machlup functional for stochastic lattice dynamical systems driven by time-varying noise}
\runtitle{OM functional for SLDSs}

\begin{aug}
	\author[A]{\inits{F.}\fnms{Xinze}~\snm{Zhang}\ead[label=e1]{xzzhang21@mails.jlu.edu.cn}},
	\author[A,B]{\inits{S.}\fnms{Yong}~\snm{Li}\ead[label=e2]{liyong@jlu.edu.cn}}\thanks{corresponding author.}.
	\address[A]{School of Mathematics, Jilin University, ChangChun, People's Republic of China}
	
	\address[B]{Center for Mathematics and Interdisciplinary Sciences,\\    ~~~ Northeast Normal University, ChangChun, People's Republic of China\\
		\printead{e1,e2}}
	
\end{aug}

\begin{abstract}
This paper investigates the Onsager-Machlup functional of stochastic lattice dynamical systems (SLDSs)  driven by time-varying noise. We extend the Onsager-Machlup functional from finite-dimensional to infinite-dimensional systems, and from constant to time-varying diffusion coefficients. We first verify the existence and uniqueness of general SLDS solutions in the infinite sequence weighted space $ l^2_{\rho} $. Building on this foundation, we employ techniques such as the infinite-dimensional Girsanov transform, Karhunen-Loève expansion, and probability estimation of Brownian motion balls to derive the Onsager-Machlup functionals for SLDSs in $ l^2_{\rho} $ space. Additionally, we use a numerical example to illustrate our theoretical findings, based on the Euler Lagrange equation corresponding to the Onsage Machup functional.
\end{abstract}



\end{frontmatter}

\section*{Statements and Declarations}
\begin{itemize}
	\item Ethical approval\\
	Not applicable.
	
	\item Competing interests\\
	The authors declare that they have no competing interests.
	
	\item Authors' contributions\\
	X. Zhang wrote the main manuscript text and prepared figures 1-3. Y. Li provided ideas and explored specific research methods. All authors reviewed the manuscript.
	
	\item Availability of data and materials\\
	Not applicable.
\end{itemize}

\section{Introduction}
Stochastic dynamical systems combine the characteristics of deterministic dynamics and stochastic processes, commonly used to model and analyze uncertainties and randomness in fields such as physics, chemistry, and biological sciences. Due to stochastic fluctuations, these systems may undergo transitions between two metastable states. Consequently, a critical problem is to determine the most probable transition path among all possible smooth paths connecting two given metastable states. A fundamental and significant tool in this context is the Onsager-Machlup functional, originally introduced by Onsager and Machlup in 1953. They defined it as a probabilistic density functional for diffusion processes with linear drift and constant diffusion coefficients. Subsequently, Tisza and Manning in 1957 extended the application of the Onsager-Machlup functional to nonlinear equations. Concurrently, Stratonovich in the same year introduced a rigorous mathematical approach, further advancing the theoretical framework of the Onsager-Machlup functional. Subsequent developments of the Onsager-Machlup functional were presented in \cite{4,10,11,12,13,14,16}, etc.

The Onsager-Machlup functional method for stochastic differential equations has seen rapid development in this century. Moret and Nualart \cite{17} investigated the Onsager-Machlup functional for stochastic differential equations driven by fractional Brownian motion. Bardina et al. \cite{18} computed the Onsager-Machlup functional for a stochastic evolution equation with additive noise using $ L^2 $-type techniques. Chao and Duan \cite{19} derived the Onsager-Machlup function for a class of stochastic dynamical systems under both Lévy noise and Brownian noise. Li and Li \cite{20} proved that the $\Gamma$-limit of the Onsager-Machlup functional is the geometric form of the Freidlin-Wentzell functional in a proper time scale. Zhang and Li \cite{24} studied the Onsager-Machlup functional for stochastic differential equations with time-varying noise of the $\alpha$-Hölder type, where $0 < \alpha < \frac{1}{4}$.

Time-varying coefficient noise provides flexibility, allowing for more accurate modeling of real-world phenomena where noise intensity changes over time. This captures dynamic characteristics that constant coefficient noise models cannot represent. This feature is especially important in fields such as financial mathematics \cite{21}, meteorological models \cite{22}, and biomedical engineering \cite{23}. Consequently, time-varying coefficient noise has unparalleled advantages over constant coefficient noise in practical applications. Despite its importance, there has been limited research addressing the Onsager-Machlup functional for stochastic differential equations driven by time-varying coefficient noise, particularly in infinite-dimensional settings.

In our recent work \cite{24}, we first examined the finite-dimensional case of the Onsager-Machlup functional for stochastic differential equations driven by time-varying coefficient noise. However, extending this result to the infinite-dimensional case presents significant challenges, and the method used in \cite{24} is not feasible in this context. Our current study aims to find new approaches to overcome these challenges. In this paper, we define the infinite sequence weighted space $ l^2_{\rho} $ as presented in Definition $\ref{D2.3}$. We investigate the Onsager-Machlup functional for the following SLDS in the $ l^2_{\rho} $ space:
\begin{equation}
	\begin{aligned}
		\frac{\mathrm{d} u_{i}(t)}{\mathrm{d} t}
		= \nu \left( u_{i -1} - 2u_{i} + u_{i + 1} \right) - \lambda u_{i} - f\left( u_{i} \right) + g_{i} + q_{i}(t) \frac{\mathrm{d} w_{i}(t)}{\mathrm{d} t}, \quad i \in \mathbb{Z}
	\end{aligned}\notag
\end{equation}
with initial condition
\begin{equation}
	\begin{aligned}
		u_{i}(0) = u_{0,i}, \quad i \in \mathbb{Z},
	\end{aligned}\notag
\end{equation}
where $u = \left( u_{i} \right)_{i \in \mathbb{Z}} \in l^{2}_{\rho}$, $\mathbb{Z}$ denotes the integer set, $\nu$ and $\lambda$ are positive constants, $f \in C^{2}(\mathbb{R})$, $g = \left( g_{i} \right)_{i \in \mathbb{Z}} \in l^{2}_{\rho}, q(t) = \left( q_i(t) \right)_{i \in \mathbb{Z}} \in l^{2}_{\rho}, q_i(t) \in C^{2}(\mathbb{R})$, and $ W_{i}: i \in \mathbb{Z} $ are independent two-side real valued standard Wiener processes. 

In Theorem \ref{T4.1} of Section 4, we present our main result, where we derive the Onsager-Machlup functional $\int_{0}^{1} OM(\varphi, \dot{\varphi}) \,{\rm d}t$ for the SLDS driven by time-varying noise. This functional can still be interpreted as providing a Lagrangian for smooth paths. The primary form of $\int_{0}^{1} OM(\varphi, \dot{\varphi}) \,{\rm d}t$ is as follows:

\begin{equation}
	\int_{0}^{1} OM(\varphi, \dot{\varphi}) \,{\rm d}t = \int_{0}^{1} \Vert Q(t)^{-1} \left({\dot{\varphi}(t) + A\varphi(t) - F(\varphi(t))} \right)\Vert^2_{\rho} \,{\rm d}t + \int_{0}^{1} { \text{Tr}_{\rho} \left(\mathcal{D}_{\varphi(t)}F \right)} \,{\rm d}t,\notag
\end{equation}
where $(A \varphi(t))_i = - \nu \left( \varphi_{i -1}(t) - 2\varphi_{i}(t) + \varphi_{i + 1}(t) \right) + \lambda \varphi_i(t)$, $(F(\varphi(t)))_i := - f(\varphi(t)_i) + g_i$ for each $\varphi(t) \in l^2_{\rho}$, $\text{Tr}_{\rho}$ represents a weighted trace and $\mathcal{D}_{\varphi(t)} F$ represents the Fréchet derivative of $F$ at $\varphi(t)$. 

The proof of this result is fraught with difficulties, primarily due to the challenges inherent in extending from finite-dimensional systems to infinite-dimensional systems, and from constant diffusion coefficients to time-varying diffusion coefficients. These difficulties are especially pronounced in the estimation of the following Taylor's expansion:
\begin{displaymath}
	\int_{0}^{1} \big\langle {{Q(t)^{-1}}{F(v(t))}}, \,{\rm d}W(t) \big\rangle_{\rho} = \int_{0}^{1} \big\langle {{Q(t)^{-1}} {\left(F(\varphi(t)) + \mathcal{D}_{\varphi(t)} (F) W^Q(t) + R(t) \right)}}, \,{\rm d}W(t) \big\rangle_{\rho},
\end{displaymath}
where $ W(t) := W(t, \omega) := \sum_{i \in \mathbb{Z}} W_i(t) e_i $ is infinite-dimensional white noise, and $ W^{Q}(t) := W^{Q}(t, \omega) := \int_{0}^{t} \langle Q(s), \mathrm{d}W(s) \rangle $ is a time-varying stochastic process taking values in $ l^2_{\rho} $.

The first term can be handled using Lemma 2.15 provided in this article. In the second term, a double integral involving Brownian motion appears. Unlike the finite-dimensional case, we need to employ techniques such as the Karhunen-Loève expansion and Skorohod integration to transform $ W^Q(t) $ into independent $\mathcal{N}(0,1)$ random variables, ultimately deriving the final result from Lemma \ref{lmma 2.9}. The third term represents higher-order remainders, which we manage using exponential inequalities for martingales and small ball estimates for Brownian motion.

Unlike the case with constant diffusion coefficients, the inverse of the diffusion coefficient operator $ Q(t)^{-1} $ appears in the Onsager-Machlup functional. This introduces significant challenges to its derivation and necessitates further consideration of relevant constraint conditions. Notably, our results are consistent with those obtained under constant coefficient noise, and the optimal paths can be determined by solving the Euler-Lagrange equations. Thus, by leveraging the Onsager-Machlup functional and varational  principles, we can numerically identify the most probable paths between states for a broader range of stochastic dynamical systems.

The structure of this paper is as follows: In Section 2, we review some basic definitions of spaces and norms; introduce the concepts of the Onsager-Machlup function and the Onsager-Machlup functional; present the Karhunen-Loève expansion for mean-square continuous stochastic processes; and introduce several important technical lemmas. In Section 3, we verify the existence, uniqueness, and tail estimates of solutions to the stochastic lattice equations. In Section 4, we derive the Onsager-Machlup function. Finally, in Section 5, we illustrate our results with a specific example.

\section{Preliminaries}
\subsection{Basic Space}
In this section, we define some basic spaces and review some basic definitions and results on approximate limits in Wiener spaces (reference \cite{1}).
\begin{definition}
	Let $l^{2}$ represent the space of squared summable real valued sequences, specifically represented as follows:
	\begin{equation}
		\begin{aligned}
			l^{2} := \left\lbrace u := \left( u_i\right)_{i \in \mathbb{Z}}  \in l^{2} ~|~ \forall t \in \mathbb{R}, ~ \sum_{i \in \mathbb{Z}} a_i^{2} < \infty \right\rbrace ,
		\end{aligned}\notag
	\end{equation}
	where $\vert \cdot \vert$ represents the absolute value.
\end{definition}
\begin{definition}
	The $l^\infty$ space consists of all bounded real valued sequences. Specifically, a sequence $\{b_i\}$ belongs to the $l^\infty$ space if and only if the absolute values of the terms in the sequence are bounded, i.e., there exists a constant $M$ such that $|b_i| \leq M$ for all $i$.
\end{definition}
In this paper, we introduce the weighted space of infinite sequences $\left( l^2_{\rho}, \Vert \cdot \Vert_{\rho} \right)$. For the sake of simplicity and without causing confusion, we will continue to use $\Vert \cdot \Vert$ to denote $\Vert \cdot \Vert_{\rho}$.
\begin{definition}\label{D2.3}
	Let $\rho_i \in l^2$, we define a separable Hilbert space $l^2_{\rho}$, specifically represented as follows:
	\begin{equation}
		\begin{aligned}
			l^{2}_{\rho} := \left\lbrace u := \left( u_i\right)_{i \in 	\mathbb{Z}}  \in l^{2}_{\rho} ~|~ \forall t \in \mathbb{R},~ \left( u_i\right)_{i \in \mathbb{Z}}  \in l^{\infty},~\text{and} ~ \sum_{i \in \mathbb{Z}} \left( \rho_i u_i\right) ^{2} < \infty \right\rbrace
		\end{aligned}\notag
	\end{equation}
	with the norm $\Vert u \Vert_{\rho} = \left( \sum_{i \in \mathbb{Z}} \left( \rho_i u_i\right)^{2} \right)^{\frac{1}{2}} $. And the inner product $\langle u, v \rangle_{\rho}$ is defined as follows:
	\begin{displaymath}
		\langle u, v \rangle_{\rho} = \sum_{i \in \mathbb{Z}} \rho_i^2 u_i v_i \quad \text{for all}~~ u, v \in l^2_{\rho}.
	\end{displaymath}
\end{definition}
\begin{remark}
	Let $u \in l^{2}_{\rho}$, we define $\rho u = \left(\rho_i u_i\right)_{i \in \mathbb{Z}}$ and we easily know that $\rho u \in l^2$. Further, we define $\rho^2 u = \left(\rho_i^2 u_i\right)_{i \in \mathbb{Z}}$ and $\rho^2 u \in l^1 \subseteq l^2$.
\end{remark}
\begin{definition}
	Let $u(t) \in L^2([0,1], l^2_{\rho}) $. The norm in this space is defined as:
	\begin{displaymath}
		\Vert u\Vert_{L^2_{\rho}} := \Vert u\Vert_{L^2([0,1], l^2_{\rho})} = \left( \int_0^1 \Vert u(t) \Vert_{\rho}^2 \, dt \right)^{\frac{1}{2}},
	\end{displaymath}
	and the inner product in $ L^2([0,1], l^2_{\rho}) $ for functions $ u(t) $ and $ v(t) $ is defined as:
	\begin{displaymath}
		\left\langle  u, v \right\rangle_{L^2_{\rho}} := \left\langle  u, v \right\rangle_{L^2([0,1], l^2_{\rho})} = \int_0^1 \left\langle  u(t), v(t)\right\rangle_{\rho} \, dt.
	\end{displaymath}
\end{definition}

Let $W = \left\{ W_t, ~t \in [0, 1] \right\}$ be a Brownian motion (Wiener process) defined in the complete filtered probability space $(\Omega, \mathcal{F}, \left\{ \mathcal{F}_t \right\}_{t \geq 0}, \mathbb{P})$. Here, $\Omega$ represents the space of continuous functions vanishing at zero, and $\mathbb{P}$ denotes the Wiener measure. Let $H := L^2([0,1], l^2_{\rho})$ be a Hilbert space and $H^1$ be the Cameron-Martin space defined as follows:
\begin{displaymath}
	H^1 := \left\{ f : [0, 1] \to \mathbb{R}^n \in H^1 ~\big|~f(0) = 0, f ~ \text{is absolutely continuous functions and} ~ f^{\prime} \in H \right\}.
\end{displaymath}
The scalar product in $H^1$ is defined as follows:
\begin{displaymath}
	\left\langle  f, g \right\rangle_{H^1} = \left\langle  f^{\prime}, g^{\prime} \right\rangle_{H} = \left\langle f^{\prime}, g^{\prime} \right\rangle_{L^2_{\rho}}
\end{displaymath}
for all $f, g \in H^1$.
Let $\mathcal{P}:H^1 \to H^1$ be an orthogonal projection with $dim \mathcal{P}H^1 < \infty$ and the specific expression
\begin{displaymath}
	\mathcal{P}h = \sum_{i = 1}^{n} \left\langle  h_i, f \right\rangle_{H^1} h_i,
\end{displaymath}
where $(h_1, ..., h_n)$ is a set of orthonormal basis in $\mathcal{P}H^1$. In addition, we can also define the
$H^1$-valued random variable
\begin{displaymath}
	\mathcal{P}^W = \sum_{i = 1}^{n} \bigg( \int_{0}^{1} {h_i^{\prime}} \,{\rm d}W_s \bigg) h_i,
\end{displaymath}
where $\mathcal{P}^W$ does not depend on $(h_1, ..., h_n)$.
\begin{definition}\label{definition 2.1}
	We say that a sequence of orthogonal projections $\mathcal{P}_n$ on $H^1$ is an approximating sequence of projections, if $dim \mathcal{P}_n H^1 < \infty$ and $\mathcal{P}_n$ converges strongly to the identity operator $I$ in $H^1$ as $n \to \infty$.
\end{definition}

\begin{definition}\label{definition 2.2}
	We say that a semi-norm $\mathcal{N}$ on $H^1$ is measurable, if there exists a random variable $\tilde{\mathcal{N}}$, satisfying $\tilde{\mathcal{N}} < \infty $ a.s, such that for any approximating sequence of projections $\mathcal{P}_n$ on $H^1$, the sequence $\mathcal{N}(\mathcal{P}^W_n)$ converges to $\tilde{\mathcal{N}}$ in probability and $\mathbb{P}(\tilde{\mathcal{N}} \leq \epsilon) > 0$ for any $\epsilon > 0$. Moreover, if $\mathcal{N}$ is a norm on $H^1$, then we call it a measurable norm.
\end{definition}

\subsection{Onsager-Machlup functional}
In the problem of finding the most probable path of a diffusion process, we encounter the issue that the probability of any single path is zero. Therefore, we search for the probability that a path is located in a certain region, which may be a tube along a differentiable function. Consequently, the Onsager-Machlup function can be defined as the Lagrangian that gives the most probable tube. We now introduce definitions of Onsager-Machlup function and Onsager-Machlup functional.
\begin{definition}
	Consider a tube surrounding a reference path $\varphi_t$ with initial value $\varphi_0 = x$ and $\varphi_t - x$ belongs to $H^1$. Assuming $\epsilon$ is given and small enough, we estimate the probability that the solution process $X_t$ is located in that tube as:
	\begin{displaymath}
		\mathbb{P} \left\{ \Vert X - \varphi\Vert \leq \epsilon\right\}  \propto C(\epsilon) {\rm exp} \left\{ -\frac{1}{2} \int_{0}^{1} {OM(t, \varphi, \dot{\varphi})} \,{\rm d}t \right\},
	\end{displaymath}
	where $\propto$ denotes the equivalence relation for $\epsilon$ small enough  and $\Vert \cdot \Vert$ is a suitable norm. Then we call the integrand $OM(t, \varphi, \dot{\varphi})$ the Onsager-Machulup function and also call integral $\int_{0}^{1} {OM(t, \varphi, \dot{\varphi})} \,{\rm d}t$ the Onsager-Machulup functional. In analogy to classical mechanics, we also refer to the Onsager-Machulup function as the Lagrangian function and the Onsager-Machulup functional as the action functional.
\end{definition}

\subsection{Karhunen-Loève expansion}
We calculate here the Karhunen-Loève expansion for a class of one-dimensional centered mean-square continuous stochastic processes that will appear in the decomposition of $W^Q(t)$.
\begin{definition}
	Assume that stochastic prcess $X : [0,1] \times \Omega \to \mathbb{R}$ is measurable
	for every $t \in [0,1]$. We say stochastic process $X(t, \omega)$ is centered if
	\begin{displaymath}
		\mathbb{E}\left[ X(t, \omega)\right]  = 0 \quad \text{for all}~ t \in [0,1].
	\end{displaymath}
	We say a stochastic process $X(t, \omega)$ is mean-square continuous if
	\begin{displaymath}
		\lim\limits_{\epsilon \to 0} \mathbb{E}\left[ \left( X(t + \epsilon, \omega) - X(t, \omega)\right)^2 \right] = 0 \quad \text{for all}~ t \in [0,1].
	\end{displaymath}
\end{definition}
For a centered mean-square continuous stochastic proces $X(t, \omega)$, we define the integral operator $K : L^2([0,1]) \to L^2([0,1])$ by
\begin{displaymath}
	(Kv)(s) := \int_{0}^{1} k(s,t)v(t)dt, \quad s,t \in [0,1],
\end{displaymath}
where $v(s) \in L^2([0,1])$ and $k(s,t) = \mathbb{E}\left[ X(s, \omega) X(t, \omega)\right] $. So we can show that $K$ is a  compact, positive and self-adjoint operator. According to the spectral theorem, $K$ has a complete set of eigenvectors $\left\lbrace l_i \right\rbrace_{i \in \mathbb{Z}} $ in $L^2([0,1])$ and real non-negative eigenvalues $\left\lbrace \lambda_i \right\rbrace_{i \in \mathbb{Z}} $:
\begin{displaymath}
	Ke_i = \lambda_i l_i.
\end{displaymath}

Next, we introduce the Karhunen-Loève expansion theorem related to this paper. For more detailed information on it, please refer to \cite{2}.
\begin{theorem}\label{theorem 2.6}
	Let $X : \Omega \times [0,1] \to \mathbb{R}$ be a centered mean-square continuous stochastic process with $X \in L^2(\Omega \times [0,1])$. Then there exists an orthonormal basis $\left\lbrace l_i\right\rbrace_{i \in \mathbb{Z}} $ of $L^2([0,1])$ such that for all $t \in D$,
	\begin{displaymath}
		X(t, \omega) = \sum_{i \in \mathbb{Z}} \lambda_i x_i(\omega) l_i(t),
	\end{displaymath}
	where the coefficients $x_i$ is a sequence of independent, standard normal $\mathcal{N}(0, 1)$ random variables and has the following expression:
	\begin{displaymath}
		x_i(\omega) = \frac{1}{\lambda_i} \int_{0}^{1} X(t, \omega)l_i(t) \mathrm{d} t,
	\end{displaymath}
	and
	\begin{displaymath}
		\lambda_i^2 = \mathrm{Var} \left[ \int_{0}^{1} X(t, \omega)l_i(t) \mathrm{d} t \right].
	\end{displaymath}
\end{theorem}

\subsection{Technical lemmas}
In this section, we will introduce several commonly utilized technical lemmas and theorems. Throughout this paper, if not mentioned otherwise, $\mathbb{E} \left(A \big| B\right)$ represents the conditional expectation of $A$ under $B$, and $C$ is a constant and will change with the line.

When we derive the Onsage-Machup functional of SDEs with additive noise, the following lemma is the most basic one, as it ensures that we handle each term separately. Its proof can be found in \cite{3}.
\begin{lemma}\label{lemma 2.7}
	For a fixed integer $N \geq 1$, let $X_1, ..., X_N \in \mathbb{R}$ be $N$ random variables defined on $(\Omega, \mathcal{F}, \left\{ \mathcal{F}_t \right\}_{t \geq 0}, \mathbb{P})$ and $\left\{D_{\epsilon}; \epsilon > 0 \right\}$ be a family of sets in $\mathcal{F}$. Suppose that for any $c \in \mathbb{R}$ and any $i = 1, ..., N$, we have
	\begin{displaymath}
		\limsup\limits_{\epsilon \to 0} \mathbb{E}\left({\rm exp}\left\{ c X_i \right\}\big|D_{\epsilon} \right) \leq 1.
	\end{displaymath}
	Then
	\begin{displaymath}
		\limsup\limits_{\epsilon \to 0} \mathbb{E}\left({\rm exp}\left\{ \sum_{i = 1}^{N}c X_i \right\} \big|D_{\epsilon} \right)= 1.
	\end{displaymath}
\end{lemma}
The following two lemmas are about the limit behavior of the expected value of independent, standard normal $\mathcal{N}(0, 1)$ random variables exponential functions, which can be referred to in \cite{4}.
\begin{lemma}\label{lmma 2.8}
	Let $(X_i)_{i \in \mathbb{Z}}$ be a sequence of independent, standard normal $\mathcal{N}(0, 1)$ random variables defined on $(\Omega, \mathcal{F}, \left\{ \mathcal{F}_t \right\}_{t \geq 0}, \mathbb{P})$ and let $(\eta_i)_{i \in \mathbb{Z}}$ and $(\theta_i)_{i \in \mathbb{Z}}$ be two real numbers sequences in $ l^2 $. Then
	\begin{displaymath}
		\lim_{\epsilon \to 0} E\left[\exp\left(\sum_{i \in \mathbb{Z}} \eta_i X_i \right) \big| \sum_{i \in \mathbb{Z}} { \theta_i^2 X_i^2} \leq \epsilon\right] = 1.
	\end{displaymath}
	Moreover, for any uniformly bounded random variable $ Y(\omega) $,
	\begin{displaymath}
		\lim_{\epsilon \to 0} E\left[\exp\left(\sum_{i \in \mathbb{Z}} Y(\omega) \eta_i X_i \right) \big| \sum_{i \in \mathbb{Z}} { \theta_i^2 X_i^2} \leq \epsilon\right] = 1.
	\end{displaymath}
\end{lemma}

\begin{lemma}\label{lmma 2.9}
	Let $(\eta_i)_{i \in \mathbb{Z}}$ and $(\theta_i)_{i \in \mathbb{Z}}$ be two real numbers sequences in $ l^2 $. And let $(X_i)_{i \in \mathbb{Z}}$ be a sequence of independent, standard normal $\mathcal{N}(0, 1)$ random variables defined on $(\Omega, \mathcal{F}, \left\{ \mathcal{F}_t \right\}_{t \geq 0}, \mathbb{P})$. Assume $ T : l^2 \to l^2$ to be a trace class operator, i.e. $\sum\limits_{i \in \mathbb{Z}} \langle T e_i, e_i \rangle < \infty$ for any orthonormal basis$ \left\lbrace e_i\right\rbrace_{i \in \mathbb{Z}} $ in space $l^2$. Then
		\begin{displaymath}
			\lim_{\varepsilon \to 0} E\left[ \exp\left( \sum_{i,j} X_i X_j T_{ij} \right) \bigg| \sum_{i \in \mathbb{Z}} { \theta_i^2 X_i^2} \leq \epsilon \right] = 1,
		\end{displaymath}
		and more generally
		\begin{displaymath}
			\lim_{\varepsilon \to 0} E\left[ \exp\left( \sum_{i,j} (X_i + \eta_i) X_j T_{ij} \right) \bigg|\sum_{i \in \mathbb{Z}} { \theta_i^2 X_i^2} \leq \epsilon \right] = 1.
		\end{displaymath}
\end{lemma}
The following theorem is about the probability estimation of Brownian motion balls, and its proof can be found in \cite{5}. Before giving this theorem, we shall assume that $\lambda_n$ is given by $\lambda_i = \lambda(i)$, where $\lambda(t) : [1, \infty) \to \mathbb{R}_+$ satisfies:
\begin{itemize}
	\item[(1)] $\lambda(t))$ is decreasing, $\lambda(t) > 0$ for all $t \geq 1$.
	\item[(2)] $\lambda(t) \leq t^{-\frac{1}{2}}$ for all $t \geq 1$.
	\item[(3)] $\int_1^\infty \lambda(t)^2 \, dt < \infty$.
	\item[(4)] $\log \lambda(t)$ is convex.
	\item[(5)] $\varphi(t) = t^{-\frac{1}{2}} \lambda(t)^{-1}$ increases to $+\infty$ on $[1, \infty)$.
\end{itemize}
\begin{lemma}\label{lemma 2.13}
	When $\lambda(n)$ satisfies $(1)$--$(5)$, let $\varphi$, $\psi$ and $H$ be given as above:
	\begin{displaymath}
		\begin{aligned}
			\varphi(x) &= x^{-\frac{1}{2}} \lambda(x)^{-1} \quad &\text{for all}~ x \geq 1,\\
			\psi(x) &= \int_{x}^{\infty} \lambda(t) \mathrm{d} t \quad &\text{for all}~ x \geq 1,\\
			H(x) &= \int_1^x \frac{t\varphi'(t)}{\varphi(t)} \, 	dt \quad &\text{for all}~ x \geq 1.
		\end{aligned}	
	\end{displaymath}
	Then for some $B > 0$, we have:
	\begin{displaymath}
		\mathbb{P}\left( \left(\sum_{i=1}^\infty \lambda(i)^2 X_i^2 \right)^{\frac{1}{2}} \leq t \right)  \geq B  x^{-\frac{1}{2}} \lambda(x) \left\{1 - \frac{\psi(x)}{t^2 - s^2}\right\} \exp\left\{-H(x + 1) - \frac{s^2}{2 \lambda(x)^2}\right\},
	\end{displaymath}
	whenever $x \geq 1$ and $\frac{1}{\varphi(x)} \leq s \leq t \leq 1$.
\end{lemma}
For the convenience of applying the above theorem, we propose the following corollary:
\begin{corollary}\label{corollary 2.11}
	If $\lambda(t) = t^{-\alpha}$ ($\alpha \geq \frac{1}{2}$), then we have
	\begin{displaymath}
		\mathbb{P}\left( \left(\sum_{i=1}^\infty \lambda(i)^2 X_i^2 \right)^{\frac{1}{2}} \leq \epsilon \right) \leq B_1 \epsilon^{\rho(1-\alpha)} \exp\left(-(\alpha - \frac{1}{2}) \epsilon^{-2\rho}\right),
	\end{displaymath}
	\begin{displaymath}
		\mathbb{P}\left( \left(\sum_{i=1}^\infty \lambda(i)^2 X_i^2 \right)^{\frac{1}{2}} \leq \epsilon \right) \geq B_2 \epsilon^{\rho(3-\alpha)} \exp\left(-\alpha (1 + \rho)^{\rho} \epsilon^{-2\rho}\right),
	\end{displaymath}
	where $\rho = \frac{1}{2\alpha - 1}$ and $B_1$ and $B_2$ are positive constants.
\end{corollary}

The following theorem is fundamental of calculating Onsage-Machup functional.
\begin{lemma}\label{lemma 2.15}
	Let $f(t) \in L^2([0, 1],l^2_{\rho})$, we have
	\begin{displaymath}
		\lim\limits_{\epsilon \to 0}\mathbb{E} \left( {\rm exp} \left\{ {\int_{0}^{1} \left\langle f(t), \,{\rm d}W(t)\right\rangle_{\rho}  } \right\} \big| \Vert W^Q(t) \Vert_{L^2_{\rho}} \leq \epsilon \right) = 1,
	\end{displaymath}
	where $W^Q(t) := \int_{0}^{t} \left\langle {e^{-As} Q(s)}, \,{\rm d}W(s)\right\rangle $.
\end{lemma}
\begin{proof}
	Since $f(t) \in L^2([0,1]; l^2_{\rho})$, we have
	\begin{displaymath}
		\rho^2 f(t) = \sum_{j=1}^{\infty} \sum_{i=1}^{\infty} \pi_{i,j} (l_{i,j} \otimes e_j)(t)
	\end{displaymath}
	with $\pi_{i,j} = \langle \rho^2 f, l_{i,j} \otimes e_j \rangle_{L^2([0,1]; l^2_{\rho})}$ and $\sum_{j=1}^{\infty} \sum_{i=1}^{\infty} \pi_{i,j}^2 < \infty$. Furthermore, we have
	\begin{displaymath}
		{\int_{0}^{1} \left\langle f(t), \,{\rm d}W(t)\right\rangle_{\rho}  } = \sum_{j=1}^{\infty} \sum_{i=1}^{\infty} \pi_{i,j} I_j (l_{i,j}) = \sum_{j=1}^{\infty} \sum_{i=1}^{\infty} \pi_{i,j} Y_{i,j},
	\end{displaymath}
	where $I_j (l_{i,j}) = \int_{0}^{1} {l_{i,j}} \mathrm{d} W_j$ and $(Y_i)_{i \in \mathbb{Z}}$ is a sequence of independent, standard normal $\mathcal{N}(0, 1)$ random variables. Hence, the result follows easily by Lemma $\ref{lmma 2.8}$.
\end{proof}

\section{The unique bounded solution of SLDSs}
In this paper, we study the Onstage-Machlup functionals for the following SLDSs:
\begin{equation}\label{1}
	\begin{aligned}
		\frac{\mathrm{d} u_{i}(t)}{\mathrm{d} t}
		= \nu \left( u_{i -1} - 2u_{i} + u_{i + 1} \right) - \lambda u_{i} - f\left( u_{i} \right) + g_{i} + q_{i}(t) \frac{\mathrm{d} w_{i}(t)}{\mathrm{d} t}, \quad i \in \mathbb{Z}
	\end{aligned}
\end{equation}
with initial condition
\begin{equation}\label{2}
	\begin{aligned}
		u_{i}(0) = u_{0,i}, \quad i \in \mathbb{Z},
	\end{aligned}
\end{equation}
where $u = \left( u_{i} \right)_{i \in \mathbb{Z}} \in l^2_{\rho}$, $\mathbb{Z}$ denotes the integer set, $\nu$ and $\lambda$ are positive constants, $f \in C^{2}(\mathbb{R})$, $g = \left( g_{i} \right)_{i \in \mathbb{Z}} \in l^2_{\rho}, q(t) = \left( q_i(t) \right)_{i \in \mathbb{Z}} \in l^2_{\rho}, q_i(t) \in C^{2}(\mathbb{R})$, and $ W_{i}: i \in \mathbb{Z} $ are independent two-side real valued standard Wiener processes. To ensure randomness, we naturally assume that $q_{i}(t) \neq 0$ .

We note that equation $\eqref{1}$ with the initial date $\eqref{2}$ is interpreted as a system of integral equations
\begin{equation}
	\begin{aligned}
		u_{i}(t) = u_{i}(0)
		+ \int_{0}^{t} \left( \nu \left( u_{i -1} - 2u_{i} + u_{i + 1} \right) - \lambda u_{i} 	- f\left( u_{i} \right) + g_{i} \right) \mathrm{d}s + \int_{0}^{t} {q_{i}(s) } \mathrm{d}w_{i}(s) , \quad i \in \mathbb{Z}.
	\end{aligned}\notag
\end{equation}

For convenience, we now formulate system $\eqref{1}$ and $\eqref{2}$ as an abstract ordinary differential equation in $l^2_{\rho}$. Denote by $B$, $B^{T}$ and $A$ the linear operators from $l^2_{\rho}$ to $l^2_{\rho}$ defined as follows. For $u = \left( u_i \right)_{i \in \mathbb{Z}} \in l^2_{\rho}$,
\begin{equation}
	\begin{aligned}
		\left( Bu \right)_{i} = u_{i + 1} - u_{i}, \quad 
		\left( B^{T}u \right)_{i} = u_{i - 1} - u_{i},
	\end{aligned}\notag
\end{equation}
and
\begin{equation}
	\begin{aligned}
		\left( Au \right)_{i} = - u_{i - 1} + 2 u_{i} - u_{i + 1} \quad \text{for all}~  i \in \mathbb{Z}.
	\end{aligned}\notag
\end{equation}
Then we find that
\begin{equation}
	\begin{aligned}
		A = B B^{T} = B^{T} B,
	\end{aligned}\notag
\end{equation}
and
\begin{equation}
	\begin{aligned}
		\left\langle B^{T}u, v \right\rangle  = \left\langle  u, Bv \right\rangle  \quad \text{for all}~ u, v \in l^2_{\rho}.
	\end{aligned}\notag
\end{equation}
Therefore the operator $ A $ is a non-negative self-adjoint operator (i.e. $\left\langle  Au, u \right\rangle_{\rho}  \geq 0$ and $\left\langle Au,v\right\rangle_{\rho} = \left\langle u,Av \right\rangle_{\rho}$ for all $u,v \in l^2_{\rho}$), and generates a $ C_0 $-semigroup $\left\lbrace \exp(tA),~ t \geq 0\right\rbrace $, which is strongly continuous.

In SLDSs, we typically assume that $f$ satisfies the following conditions:
\begin{itemize}
	\item[(f1)] The monotonically non-decreasing condition:
	\begin{displaymath}
		f(0) = 0 \quad \text{and} \quad \left( f(x) - f(y) \right)\left( x - y \right) \geq 0 \quad \text{for all}~ x, y \in \mathbb{R}.
	\end{displaymath}
	\item[(f2)] The polynomial growth condition:
	\begin{displaymath}
		\left| f(x) \right| \leq C_{f} \left| x \right| \left( 1+ x^{2p} \right) \quad \text{for all}~ x \in \mathbb{R},
	\end{displaymath}
	where $p$ is a positive integer.
\end{itemize}

Let $\tilde{f}$ be the Nemytski operator associated with $f$, that is, for $u = \left( u_{i} \right)_{i \in \mathbb{Z}} \in l^2_{\rho}$, let $\tilde{f}(u) = \left( f(u_i) \right)_{i \in \mathbb{Z}}$. Due to $\left|f(u_i)\right| = \left|f(u_i) - f(0)\right| = \left| {f}^{\prime}(\xi_{i}) \right| \left| u_{i} \right|, f \in C^2$ and $u_i \in l^2_{\rho}$, we easily obtain $\tilde{f}(u) \in l^2_{\rho}$. In addition,
\begin{equation}
	\begin{aligned}\label{3}
		\Vert \tilde{f}(u)\Vert^{2}_{\rho} = \Vert \tilde{f}(u) - \tilde{f}(0) \Vert^{2}_{\rho} = \sum_{i \in \mathbb{Z}} \left| {f}^{\prime}(\xi_{i}) \right|^{2} \left| \rho_i u_{i} \right|^{2},
	\end{aligned}
\end{equation}
it follows that there exists a constant $\mu$ (depending on $f$) such that
\begin{equation}
	\begin{aligned}
		\Vert \tilde{f}(u)\Vert^{2}_{\rho} \leq  \mu \Vert u \Vert^{2}_{\rho},
 	\end{aligned}\notag
\end{equation}
which means $\Vert \tilde{f}(u) \Vert_{\rho} < \infty$.
Similar to $\eqref{3}$, it can be observed that $\tilde{f}$ is locally Lipschitz from $l^2_{\rho}$ to $l^2_{\rho}$. More precisely, for every bounded set $Y$ in $l^2_{\rho}$, there exists a constant $C_{Y}$, depending solely on $Y$, such that
\begin{equation}
	\begin{aligned}
		\Vert \tilde{f}(x) - \tilde{f}(y) \Vert^{2}_{\rho} \leq C_{Y} \Vert x - y \Vert^{2}_{\rho}, \quad \forall x,y \in Y.
	\end{aligned}\notag
\end{equation}
In the sequel, when no confusion arises we identify $\tilde{f}$ with $f$.

In addition, let $e^{i} \in l^2_{\rho}$ denote the element having $1$ at position $i$ and all the other components $0$, so $\left\lbrace e^{i}\right\rbrace_{i \in \mathbb{Z}}$ is a set of orthogonal bases of $l^2_{\rho}$. Define
\begin{equation}
	\begin{aligned}
		W(t):= W(t, \omega) := \sum_{i \in \mathbb{Z}} W_i(t) e_i
	\end{aligned}\notag
\end{equation}
to be white noise with values on the probability space $\left( \Omega, \mathcal{F}, \mathbb{P} \right)$, where
\begin{equation}
	\begin{aligned}
		\Omega = \left\lbrace \omega \in C(\mathbb{R}, l^{2}_{\rho}): \omega(0) = 0 \right\rbrace  
	\end{aligned}\notag
\end{equation}
is endowed with the compact open topology (see \cite{6}, Appendix A.2), $\mathbb{P}$ is the corresponding Wiener measure and $\mathcal{F}$ is the $\mathbb{P}$-completion of the Borel $\sigma$-algebra on $\Omega$. Let
\begin{equation}
	\begin{aligned}
		\theta(t) \omega(\cdot) = \omega(\cdot + t) - \omega(t), \quad t \in \mathbb{R},
	\end{aligned}\notag
\end{equation}
then $\left( \Omega, \mathcal{F}, \mathbb{P}, \left( \theta_{t} \right)_{t \in \mathbb{R}} \right)$ is a metric dynamical system with the filtration
\begin{equation}
	\begin{aligned}
		\mathcal{F}_{t} := \bigcup_{s \leq t} \mathcal{F}^{t}_{s},\quad t \in \mathbb{R},
	\end{aligned}\notag
\end{equation}
where
\begin{equation}
	\begin{aligned}
		\mathcal{F}^{t}_{s} = \sigma\left\lbrace W(\tau_{1}) - W(\tau_{2}): s \leq \tau_{1} \leq \tau_{2} \leq t \right\rbrace
	\end{aligned}\notag
\end{equation}
is the smallest $\sigma$-algebra generated by the random variable $W(\tau_{2}) - W(\tau_{1})$ for all $\tau_{1}, \tau_{2}$ such that
$s \leq \tau_{1} \leq \tau_{2} \leq t$. Note that $ \theta_{\tau}^{-1} \mathcal{F}^{t}_{s} = \mathcal{F}^{t + \tau}_{s + \tau}$, so $\left( \theta_{t} \right)_{t \in \mathbb{R}}$ is filtered with respect to $\mathcal{F}^{t}_{s}$ (see \cite{6} for
more details).

Let operator $Q(t)$ is diagonal when expressed in the orthonormal basis $\left\lbrace e^{i}\right\rbrace_{i \in \mathbb{Z}}$ of $l^2_{\rho}$, and its diagonal element is $\left( q_i(t) \right)_{i \in \mathbb{Z}}$. We have that
\begin{equation}\label{4}
	W^{Q}(t) := W^{Q}(t,\omega) := \int_{0}^{t} \left\langle  Q(s) , \mathrm{d}W(s)\right\rangle
\end{equation}
is time-varying stochastic process with values in $l^2_{\rho}$. For any $T > 0$, the series $\sum_{i \in \mathbb{Z}} q_{i}(t) w_{i}$ from $\eqref{4}$ converges to $W^{Q}$ in the space $C\left( \left[ 0, T \right] , l^2_{\rho} \right) $ for a.e. $\omega \in \Omega$ (see \cite{8} Theorem 4.3). We can assume without loss of generality that $W^{Q}(\cdot, \omega)$ is continuous for all $\Omega_{1}$ with $\mathbb{P}(\Omega_{1}) = 1$.  Additionally, $W^{Q}(t)$ is almost surely bounded in the finite time interval $[0,T]$.

The equation $\eqref{1}$ with initial condition $\eqref{2}$ may be rewritten as an equation in $l^2_{\rho}$
\begin{equation}\label{5}
	\begin{aligned}
		\frac{\mathrm{d} u(t)}{\mathrm{d} t}
		= -\nu Au(t)  - \lambda u(t) - f\left( u(t) \right) + g + Q(t) \frac{\mathrm{d} W(t)}{\mathrm{d} t}, \quad t \in [0,T]
	\end{aligned}
\end{equation}
with initial condition
\begin{equation}\label{6}
	\begin{aligned}
		u(0) = \left( u_{i}(0) \right)_{i \in \mathbb{Z}} \in l^2.
	\end{aligned}
\end{equation}
And it has the integral expression
\begin{equation}\label{7}
	\begin{aligned}
		u(t) = u(0)
		+ \int_{0}^{t} \left( - \nu A u(s) - \lambda u(s) - f\left( u(s) \right) + g \right) \mathrm{d}s + W^{Q}(t) , \quad t \geq 0,\quad \omega \in \Omega.
	\end{aligned}
\end{equation}

By excluding a set measure zero, we can take $W^Q(t, \omega)$ to be continuous on $\left[ 0, \infty \right)$ and bounded in the finite time interval [0,T] for all $\omega \in \Omega_{1}$, where $\Omega_{1}$ has full measure. Below we provide the definition of a metric dynamical system:
\begin{definition}
	A system $(\Omega, \mathcal{F}, \mathbb{P}, (\theta_t)_{t \in \mathbb{R}})$ is called a metric dynamical system if $\theta : \mathbb{R} \times \Omega \to \Omega$ is $(\mathcal{B}(\mathbb{R}) \times \mathcal{F}, \mathcal{F})$-measurable, $\theta_0$ is the identity on $\Omega$, $\theta_{s+t} = \theta_s \circ \theta_t$ for all $s,t \in \mathbb{R}$, and $\theta_t \mathbb{P} = \mathbb{P}$ for all $t \in \mathbb{R}$.
\end{definition}
Next, we will briefly prove that equation $\eqref{1}$ has a unique weak solution that continuously depends on the initial conditions, and generates a continuous stochastic dynamical system. Our proof idea is mainly inspired by [6].
\begin{theorem}\label{T3.1}
	Let $T > 0$, and assume $f$ satisfies condition $(\mathrm{f}1)$ and condition $(\mathrm{f}2)$. Then the following statements hold:
	\begin{itemize}
		\item[(1)] Equation $\eqref{7}$ admits a solution $u \in L^2\left(\Omega_{1}, C\left( \left[ 0, T \right] , l^2_{\rho} \right) \right) $, which is almost surely unique.
		\item[(2)] For a.e.$\omega \in \Omega$ we have the following estimate
		\begin{equation}\label{8}
			\begin{aligned}
				\sup_{t \in \left[ 0, T \right]} \Vert u(t) \Vert^{2}_{\rho} \leq C\left( \Vert u_{0} \Vert^{2}_{\rho} + \sup_{t \in \left[ 0, T \right]} \Vert W^{Q}(t) \Vert^{2}_{\rho} + \int_{0}^{T} {\left( \Vert W^{Q}(t) \Vert^{2}_{\rho} + \Vert W^{Q}(t) \Vert^{4p + 2}_{\rho} + \Vert g \Vert^{2}_{\rho} \right)}  \mathrm{d}t \right) ,
			\end{aligned}
		\end{equation}
		where $C > 0$ is a constant.
		\item[(3)] The solution $u$ of $\eqref{7}$ depends continuously on the initial data $u_{0}$.
	\end{itemize}
\end{theorem}
\begin{proof}
	Let $v(t) = u(t) - W^Q(t)$. Equation $\eqref{7}$ has a solution $u \in L^2(\Omega_{1}, C[0,T], l^2_{\rho})$ for all $\omega$ in $\Omega_{1}$ if and only if the following equation:
	\begin{equation}\label{9}
		\begin{aligned}
			v(t) = u(0)
			+ \int_{0}^{t} \left( - \nu A v(s) - \lambda v(s) - f\left( v(s) + W^Q(s) \right) + g  - v A W^Q(s) - \lambda W^Q(s) \right) \mathrm{d}s
		\end{aligned}
	\end{equation}
	has a solution $v \in L^2(\Omega_{1}, C[0,T], l^2_{\rho})$ for all $\omega$ in $\Omega_{1}$ and $t \in [0,T]$. When $\omega \in \Omega_{1}$ is fixed, equation $\eqref{9}$ is a deterministic equation. Based on the standard argument, since the function $f$ is locally Lipschitz, equation $\eqref{9}$ has a local solution $v \in L^2(\Omega_{1},C[0,T], l^2_{\rho})$. We now demonstrate that the local solution is a global solution. Let $\omega \in \Omega_{1}$, by calculating the integral of $\frac{\mathrm{d}v(t)^2}{\mathrm{d}t}$ over the interval $[0,t]$, we obtain
	\begin{equation}
		\begin{aligned}
			\Vert v(t) \Vert^2_{\rho} &= \Vert u(0) \Vert^2_{\rho}
			+ \int_{0}^{t} \bigg[ \left\langle  - \nu A v(s),v(s) \right\rangle_{\rho} - \lambda \Vert v(s) \Vert^2_{\rho} - \left\langle f\left( v(s) + W^Q(s) \right) - f\left(  W^Q(s) \right) ,v(s) \right\rangle_{\rho} \\
			& \quad + \left\langle g  - \nu A W^Q(s) - \lambda W^Q(s) - f\left(  W^Q(s) \right),v(s) \right\rangle_{\rho} \bigg] \mathrm{d}s\\	
			&\leq \Vert u(0) \Vert^2_{\rho} + C_0 \int_{0}^{t} {\Vert W^Q(s) \Vert^2_{\rho} + \Vert W^Q(s) \Vert^{4p + 2}_{\rho} + \Vert g(s) \Vert^2_{\rho}} \mathrm{d}s,
		\end{aligned}\notag
	\end{equation}
	where $C_0$ is a positive constant depending on $\nu$, $\lambda$, $C_f$, and $A$. Then it is established that $v(t)$ is bounded by a continuous function, that is, it is bounded and does not exhibit divergent behavior in any finite time $[0, T]$. This ensures that we can use extension techniques to extend from local solutions to global solutions.
	
	Setting $t = T$, from the above equation, we obtain
	\begin{equation}\label{10}
		\begin{aligned}
			\sup_{t \in [0,T]} \Vert v(t) \Vert^2_{\rho} \leq \Vert u(0) \Vert^2_{\rho} + C_0 \int_{0}^{T} {\Vert W^Q(s) \Vert^2_{\rho} + \Vert W^Q(s) \Vert^{4p + 2}_{\rho} + \Vert g(s) \Vert^2_{\rho}} \mathrm{d}s \quad \text{for all}~~ \omega \in \Omega_{1}.
		\end{aligned}
	\end{equation}
	By taking the expectation on both sides of equation $\eqref{10}$, and due to the quadratic variation property of stochastic integrals, 
	\begin{equation}
		\begin{aligned}
			\mathbb{E} \Vert W^Q(s) \Vert^2_{\rho} = \mathbb{E}\left( \sum_{i \in \mathbb{Z}} \left( \int_{0}^{t} {\rho}_i q_i(s)  \mathrm{d}W_i(s) \right)^2\right)  = \int_{0}^{t}  \Vert Q(s)\Vert^2_{\rho}  \mathrm{d}t < \infty,
		\end{aligned}\notag
	\end{equation}
	then we can prove that $v \in L^2\left(\Omega_{1}, C\left( \left[ 0, T \right] , l^{2}_{\rho} \right) \right)$. Hence, equation $\eqref{7}$ admits a global solution $u \in L^2(\Omega_{1}, C([0, T], l^2_{\rho}))$, and from $\eqref{10}$, it follows that the inequality 
	\begin{equation}
		\begin{aligned}
			\sup_{t \in \left[ 0, T \right]} \Vert u(t) \Vert^{2}_{\rho} \leq C\left( \Vert u_{0} \Vert^{2}_{\rho} + \sup_{t \in \left[ 0, T \right]} \Vert W^{Q}(t) \Vert^{2}_{\rho} + \int_{0}^{T} {\left( \Vert W^{Q}(t) \Vert^{2}_{\rho} + \Vert W^{Q}(t) \Vert^{4p + 2}_{\rho} + \Vert g \Vert^{2}_{\rho}\right)}  \mathrm{d}s \right)
		\end{aligned}\notag
	\end{equation}
	is satisfied for almost every $\omega \in \Omega$, where $C = \max\left\{2 , C_0 \right\}$.
	
	Let $u^1_0, u^2_0 \in l^2_{\rho}$ and define $u^1(t) := u(t, u^1_0)$ and $u^2(t) := u(t, u^2_0)$ as the respective solutions of equation $\eqref{7}$. Then
	\begin{equation}
		\begin{aligned}
			\Vert u^1(t) - u^2(t) \Vert^2_{\rho} &= \Vert u^1_0 - u^2_0 \Vert^2_{\rho} + 2 \int_{0}^{t} \big( -\nu \left\langle  A(u^1(s) - u^2(s)), u^1(s) - u^2(s) \right\rangle_{\rho} -\lambda \Vert u^1(s) - u^2(s) \Vert^2_{\rho} \\ 
			& \quad -\left\langle \left( f(u^1(s)) - f(u^2(s)) \right), u^1(s) - u^2(s) \right\rangle_{\rho}   \big)  \mathrm{d}s.
		\end{aligned}\notag
	\end{equation}
	By using the properties of $A$ and $f$ we have
	\begin{equation}
		\begin{aligned}
			\Vert u^1(t) - u^2(t) \Vert^2_{\rho} \leq  \Vert u^1_0 - u^2_0 \Vert^2_{\rho} \quad \text{for all}~~ t \in [0,T].
		\end{aligned}\notag
	\end{equation}
	Therefore,
	\begin{equation}
		\begin{aligned}
			\sup_{t \in [0,T]} \Vert u^1(t) - u^2(t) \Vert^2_{\rho} \leq  \Vert u^1_0 - u^2_0 \Vert^2_{\rho}.
		\end{aligned}\notag
	\end{equation}
	If $u^1_0 = u^2_0$, the above inequality demonstrates that the solution of equation $\eqref{7}$ is almost surely unique and continuously dependent on the initial data. Therefore, statements $(1)$, $(2)$ and $(3)$ of this theorem are satisfied.
	\end{proof}

	\begin{theorem}\label{T3.2}
		Equation $\eqref{7}$ generates a continuous random dynamical system $\{\phi(t)\}_{t \geq 0}$ on $(\Omega, \mathcal{F}, \mathbb{P}, \{\theta_t\}_{t \in \mathbb{R}})$, where
		\begin{displaymath}
		\phi(t, \omega, u_0) = u(t, \omega, u_0) 
		\end{displaymath}
	for all $t \geq 0$ and almost every $\omega \in \Omega$.
	\end{theorem}
	\begin{proof}
	Theorem $\ref{T3.1}$ states that for all $\omega \in {\Omega}_1$, $u(\cdot, u_0)$ exhibits continuous dependence on the initial data $u_0$. So a continuous modification of $u(t, u_0)$ exists, denoted by $\{\phi(t)\}_{t\geq 0}$, ensuring that $\phi(\cdot, \omega, \cdot): [0, \infty) \times l^2_{\rho} \rightarrow l^2_{\rho}$ is continuous for every $\omega$, and $\phi$ satisfies equation $\eqref{7}$ for almost every $\omega$.

	From the definition of $\left( \theta_{t} \right)_{t \in \mathbb{R}} $ we have the following property:
	\begin{equation}\label{11}
		\begin{aligned}
			W^Q(s + t, \omega) &= \int_{0}^{s + t}  \left\langle Q(r) , \mathrm{d}w(r, \omega)\right\rangle \\
			& = \int_{0}^{s}  \left\langle Q(r, \omega) , \mathrm{d}w(r, \omega)\right\rangle + \int_{s}^{s + t}  \left\langle Q(r) , \mathrm{d}w(r, \omega)\right\rangle \\
			& = \int_{0}^{s} \left\langle  Q(r, \omega) , \mathrm{d}w(r, \omega) \right\rangle + \int_{0}^{t} \left\langle  Q(r + s) , \mathrm{d}w(r,\theta_{s}\omega)\right\rangle \\
			& = W^Q(s, \omega) + W^Q(t, \theta_{s}\omega).
		\end{aligned}
	\end{equation}
	For simplicity, let $F(u) = -\nu Au - \lambda u - f(u) + g$ for each $u \in l^2_{\rho}$. Using equation $\eqref{7}$, we obtain
	\begin{equation}\label{12}
		\begin{aligned}
			\phi (t, \theta_{s} \omega ,\phi(s, \omega, u_0)) &= \phi(s, \omega, u_0) + \int_{0}^{t} {F\left( \phi (r, \theta_{s} \omega ,\phi(s, \omega, u_0)) \right) } \mathrm{d}r + W^Q(t, \theta_{s}\omega)\\
			&= u_0 + \int_{0}^{s} {F\left( \phi (r, \omega, u_0) \right) } \mathrm{d}r + W^Q(s, \omega) \\
			&\quad + \int_{0}^{t} {F\left( \phi (r, \theta_{s} \omega ,\phi(s, \omega, u_0)) \right) } \mathrm{d}r + W^Q(t, \theta_{s}\omega)\\
			&= u_0 + \int_{0}^{s} {F\left( \phi (r, \omega, u_0) \right) } \mathrm{d}r + W^Q(s, \omega) \\
			&\quad+ \int_{s}^{s + t} {F\left( \phi (r-s, \theta_{s} \omega ,\phi(s, \omega, u_0)) \right) } \mathrm{d}r + W^Q(t, \theta_{s}\omega).
		\end{aligned}
	\end{equation}
	For each $\omega \in \Omega$, we define
	\begin{displaymath}
		\Phi(r, \omega, u_0) =\left\{
		\begin{aligned}
			&\phi (r, \omega, u_0), \quad &\text{if}~~ 0\leq r \leq s,\\
			&\phi (r-s, \theta_{s} \omega ,\phi(s, \omega, u_0)), \quad &\text{if}~ ~s \leq r.
		\end{aligned}
		\right.
	\end{displaymath}
	Setting $r = t + s$, we have
	\begin{displaymath}
		\Phi(r, \omega, u_0) = \phi (t, \theta_{s} \omega ,\phi(s, \omega, u_0)) \quad \text{for all}~~ s,t \geq 0.
	\end{displaymath}
	From equations $\eqref{11}$ and $\eqref{12}$, it follows that
	\begin{equation}\label{13}
		\Phi(t + s, \omega, u_0) = u_0 + \int_{0}^{t + s} {F\left( \phi (r, \omega, u_0) \right) } \mathrm{d}r + W^Q(t + s, \omega).
	\end{equation}
	Given that both $\phi$ and $\Phi$ are solutions to equation $\eqref{7}$, and considering the uniqueness of this solution, it follows that
	\begin{displaymath}
		\Phi(t + s, \omega, u_0) = \phi(t + s, \omega, u_0) \quad \text{for all}~~ s,t \geq 0 ~~\text{and}~~ \omega \in \Omega_1.
	\end{displaymath}
	Therefore, from $\eqref{13}$ we can deduce
	\begin{displaymath}
		\phi(t + s, \omega, u_0) = \phi(t, \theta_{s} \omega,\phi(s, \omega, u_0)) \quad \text{for all}~~ s,t \geq 0.
	\end{displaymath}
	\end{proof}

	To perform numerical simulations of infinite-dimensional SLDSs, we need a finite-dimensional approximation. For every integer $n \geq 1$, let $\mathbb{R}^{2n + 1}$ be an $(2n+1)$-dimensional Euclidean space. Then the approximate finite dimensional system can be written as a $(2n + 1)$-dimensional ordinary differential equation as follows:
	\begin{equation}\label{14}
		\begin{aligned}
			\frac{\mathrm{d} u_{-n}(t)}{\mathrm{d} t}
			&= \nu \left( u_{n} - 2u_{-n} + u_{-n + 1} \right) - \lambda u_{-n} - f\left( u_{-n} \right) + g_{-n} + q_{-n}(t) \frac{\mathrm{d} w_{-n}(t)}{\mathrm{d} t},\\
			\frac{\mathrm{d} u_{-n + 1}(t)}{\mathrm{d} t}
			&= \nu \left( u_{-n} - 2u_{-n + 1} + u_{-n + 2} \right) - \lambda u_{-n + 1} - f\left( u_{-n + 1} \right) + g_{-n + 1} + q_{-n + 1}(t) \frac{\mathrm{d} w_{-n + 1}(t)}{\mathrm{d} t},\\
			\vdots\\
			\frac{\mathrm{d} u_{n - 1}(t)}{\mathrm{d} t}
			&= \nu \left( u_{n - 2} - 2u_{n - 1} + u_{n} \right) - \lambda u_{n - 1} - f\left( u_{n - 1} \right) + g_{n - 1} + q_{n - 1}(t) \frac{\mathrm{d} w_{n - 1}(t)}{\mathrm{d} t},\\
			\frac{\mathrm{d} u_{n}(t)}{\mathrm{d} t}
			&= \nu \left( u_{n - 1} - 2u_{n} + u_{-n} \right) - \lambda u_{n} - f\left( u_{n} \right) + g_{n} + q_{n}(t) \frac{\mathrm{d} w_{n}(t)}{\mathrm{d} t}\\
		\end{aligned}
	\end{equation}
	with the initial date
	\begin{equation}
		\begin{aligned}
			\left( u_{-n}, \cdots, u_{n} \right) (0) = \left( u_{0,-n}, \cdots, u_{0,n} \right) \in \mathbb{R}^{2n + 1}.
		\end{aligned}\notag
	\end{equation}

	According to equation (1.3), define the matrix $\tilde{A}$ as a $(2n + 1) \times (2n + 1)$ tridiagonal matrix with the following structure:
	\begin{equation*}
		\tilde{A} = 
		\begin{pmatrix}
			2 & -1 & 0 & \cdots & 0 & -1 \\
			-1 & 2 & -1 & \cdots & 0 & 0 \\
			0 & -1 & 2 & \cdots & 0 & 0 \\
			\vdots & \vdots & \vdots & \ddots & \vdots & \vdots \\
			0 & 0 & 0 & \cdots & 2 & -1 \\
			-1 & 0 & 0 & \cdots & -1 & 2
		\end{pmatrix}_{(2n + 1)\times (2n + 1)},
	\end{equation*}
	where the main diagonal entries are $2$, the first super- and sub-diagonal entries are $-1$, and the upper right and lower left corners also contain $-1$. All other entries are $0$.
	
	Additionally, define the matrix $\tilde{B}$ as a $(2n + 1) \times (2n + 1)$ matrix given by:
	\begin{equation*}
		\tilde{B} = 
		\begin{pmatrix}
			-1 & 1 & 0 & \cdots & 0 & 0 \\
			0 & -1 & 1 & \cdots & 0 & 0 \\
			0 & 0 & -1 & \cdots & 0 & 0 \\
			\vdots & \vdots & \vdots & \ddots & \vdots & \vdots \\
			0 & 0 & 0 & \cdots & -1 & 1 \\
			1 & 0 & 0 & \cdots & 0 & -1
		\end{pmatrix}_{(2n + 1)\times (2n + 1)},
	\end{equation*}
	where the main diagonal entries are $-1$, the first super-diagonal entries are $1$, and the lower-left corner also contains $1$. All other entries are $0$. It follows that $\tilde{A} = \tilde{B}\tilde{B}^{T} = \tilde{B}^{T}\tilde{B}$, where $\tilde{B}^{T}$ denotes the transpose of $\tilde{B}$. 
	
	Then the above system of equations $\eqref{14}$ is equivalent to the following equation:
	\begin{equation}\label{15}
		\begin{aligned}
			\frac{\mathrm{d} \tilde{u}(t)}{\mathrm{d} t}
			= F(\tilde{u}) + \tilde{Q}(t) \frac{\mathrm{d} \tilde{W}(t)}{\mathrm{d} t}, \quad t \in [0,T]
		\end{aligned}
	\end{equation}
	with initial condition
	\begin{equation}\label{16}
		\begin{aligned}
			\tilde{u}(0) = \tilde{u}_{0} = \left( u_{i}(0) \right)_{|i| < n} \in \mathbb{R}^{2n + 1},
		\end{aligned}
	\end{equation}
	where $\tilde{F}(\tilde{u}) := \nu \tilde{A} \tilde{u}  - \lambda \tilde{u} - f\left( \tilde{u} \right) + \tilde{g}$, vector functions $\tilde{u} = \left( u_{i} \right)_{\left| i \right| \leq n } $, $ f\left( \tilde{u} \right)  = \left( f\left( u_{i} \right) \right)_{\left| i \right| \leq n } $, $\tilde{g} = \left( g_{i} \right)_{\left| i \right| \leq n }$, $\tilde{W} = \left(w_{i} \right)_{\left| i \right| \leq n } $ and $\tilde{Q}(t)$ is the $(2n + 1) \times (2n + 1)$ function matrix
	\begin{equation*}
		\begin{array}{c}
			\tilde{Q}(t)
		\end{array}
		=
		\left(
		\begin{array}{ccccccc}
			q_{-n}(t) & 0 & 0 & 0 & \cdots & 0 & 0\\
			0 & q_{-n+1}(t) & 0 & 0 & \cdots & 0 & 0 \\
			0 & 0 & q_{-n+2}(t) & 0 & \cdots & 0 & 0 \\
			\vdots & \vdots & \vdots & \vdots & \vdots & \vdots & \vdots\\
			0 & 0 & 0 & 0 & \cdots & q_{n-1}(t) & 0 \\
			0 & 0 & 0 & 0 & \cdots & 0 & q_{n}(t) \\
		\end{array}
		\right)_{(2n + 1)\times (2n + 1) } \hspace{-5em}.
	\end{equation*}
	Obviously, equation $\eqref{15}$ with initial condition $\eqref{16}$ is wellposed in $\mathbb{R}^{2n + 1}$, that is, for every $\tilde{u}_{0} \in \mathbb{R}^{2n+1}$, there exists a unique solution $\tilde{u}(t) \in L^{2} \left(\Omega, C \left( \left[ 0, T \right], \mathbb{R}^{2n+1} \right) \right)$. Thus, we have the existence of a dynamical system $\left\lbrace S_n(t)\right\rbrace_{t \leq 0}$ which maps $\mathbb{R}^{2n+1}$ to $\mathbb{R}^{2n+1}$ defined for $\tilde{u}_{0} \in \mathbb{R}^{2n+1}$ by $S_n(t) \tilde{u}_{0} = \tilde{u}(t)$, the solution of equation $\eqref{15}$ with initial condition $\eqref{16}$.
	\begin{theorem}\label{T3.3}
		Assume $f$ satisfies condition $(\mathrm{f}1)$ and condition $(\mathrm{f}2)$, then the solution $\tilde{u}(t)$ of the (2n+1)-dimensional equation $\eqref{15}$ with initial condition $\eqref{16}$ converges on the mean square to the solution $u(t)$ of the infinite dimensional equation $\eqref{5}$ with initial condition $\eqref{6}$ as $n \to \infty$, that is,
		\begin{equation}\label{17}
			\lim\limits_{n \to \infty} \mathbb{E} \left( \Vert u(t) - \tilde{u}(t) \Vert_{\rho} \right)^{2} \to 0.
		\end{equation}
	\end{theorem}
	\begin{proof}
		To prove that $\eqref{17}$ holds is equivalent to prove that for every $\epsilon > 0$, there exist $K(\epsilon)$ such that the solution $u(t)$ of the infinite dimensional equation $\eqref{1}$ satisfies
		\begin{equation}\label{18}
			\sum_{|i| > K(\epsilon)} \mathbb{E} \left(  {\rho}_i u_i(t) \right) ^2 < \epsilon \quad \text{for all}~~ t \in [0,T],
		\end{equation}
		where $K(\epsilon)$ depend on $\nu$, $\lambda$, $g$ and $Q$.
		
		Choose a smooth function $\theta(t)$ satisfying $ 0 \leq \theta(t) \leq 1 $ for all $ t \in \mathbb{R}_+ $, with
		\begin{displaymath}
			\theta(t) =\left\{
			\begin{aligned}
				&0, \quad &\text{if}~~ 0 \leq t \leq 1,\\
				&1, \quad &\text{if}~~ t \geq 2.
			\end{aligned}
			\right.
		\end{displaymath}
		Additionally, let there exist a positive constant $ C_{\theta} $ such that $ |\theta'(t)| \leq C_{\theta} $ for all $ t \in \mathbb{R}_+ $.
		
		Let $ k $ be a fixed integer to be specified later, and by taking the inner product of the differential form of equation $\eqref{1}$ with $ \theta \left(\frac{|i|}{k}\right) {\rho}_i^2 u_i(t) $ in $ l^2_{\rho} $, we obtain
		\begin{equation}\label{19}
			\begin{aligned}
				\frac{1}{2} \frac{\mathrm{d}}{\mathrm{d}t} \sum_{i} \theta \left(\frac{|i|}{k}\right) |{\rho}_i u_i(t)|^2 
				&= - \nu \left\langle Bu(t), B\theta \left(\frac{|i|}{k}\right) u(t) \right\rangle_{\rho} - \lambda\sum_{i} \theta \left(\frac{|i|}{k}\right) |{\rho}_i u_i(t)|^2 - \sum_{i} \theta \left(\frac{|i|}{k}\right) f(u_i(t)){\rho}_i^2 u_i(t)\\
				& + \sum_{i}\theta \left(\frac{|i|}{k}\right) g_i {\rho}_i^2 u_i(t) + \sum_{i}\theta \left(\frac{|i|}{k}\right) q_i(t) {\rho}_i^2 u_i(t) \frac{\mathrm{d}w_i}{\mathrm{d}t}.
			\end{aligned}
		\end{equation}
		Now we estimate the right-hand side of $\eqref{19}$. First we have
		\begin{align*}
				- \left\langle Bu(t), B\theta \left(\frac{|i|}{k}\right) u(t) \right\rangle_{\rho} &= - \rho_i^2 \sum_{i} (u_{i+1}(t) - u_{i}(t))(\theta \left(\frac{|i|}{k}\right)u_{i+1}(t) - \theta \left(\frac{|i|}{k}\right)u_{i}(t)) \\
				&= -\rho_i^2 \sum_{i} \left( \theta\left(\frac{|i + 1|}{k}\right) - \theta\left(\frac{|i|}{k}\right) \right) (u_{i+1}(t) - u_{i}(t))u_{i+1}(t)\\
				&\quad  - \rho_i^2\sum_{i} \theta\left(\frac{|i|}{k}\right)(u_{i+1}(t) - u_{i}(t))^2 \\
				&\leq \rho_i^2 \bigg| \sum_{i} \left( \theta\left(\frac{|i + 1|}{k}\right) - \theta\left(\frac{|i|}{k}\right) \right) (u_{i+1}(t) - u_{i}(t))u_{i+1}(t) \bigg| \\
				&\leq \rho_i^2 \sum_{i} \frac{\left| \theta'(\xi_{i}) \right|}{k} |u_{i+1}(t) - u_{i}(t)| |u_{i+1}(t)| \\
				&\leq \frac{2C_{\theta}}{k} \Vert u(t) \Vert^2_{\rho}.
		\end{align*}
		From statement $(2)$ of Theorem $\ref*{T3.1}$, it is known that $\Vert u(t) \Vert^2_{\rho}$ is bounded on the finite interval $[0, T]$ for a.e. $\omega \in \Omega$. Here, assuming the upper bound is $M$, we have
		\begin{equation}\label{20}
			\begin{aligned}
				- \left\langle  Bu(t), B\theta \left(\frac{|i|}{k}\right) u(t) \right\rangle_{\rho}  
				\leq \frac{2MC_{\theta}}{k}, \quad \text{for all}~~ t \in [0,T].
			\end{aligned}
		\end{equation}
		On the other hand, we have
		\begin{equation}\label{21}
			\begin{aligned}
				\sum_{i} \theta\left(\frac{|i|}{k}\right) g_i \rho_i^2 u_i(t) &= \sum_{|i| \geq k} \theta\left(\frac{|i|}{k}\right) g_i \rho_i^2 u_i(t) \\
				&\leq \frac{1}{2} \lambda \sum_{|i| \geq k} \theta^2\left(\frac{|i|}{k}\right) |\rho_i u_i(t)|^2 + \frac{1}{2\lambda} \sum_{|i| \geq k} |\rho_i g_i|^2 \\
				&\leq \frac{1}{2} \lambda \sum_{i} \theta\left(\frac{|i|}{k}\right) |\rho_i u_i(t)|^2 + \frac{1}{2\lambda} \sum_{|i| \geq k} |\rho_i g_i|^2.
			\end{aligned}
		\end{equation}
		Therefore, based on $\eqref{19}$, $\eqref{20}$, $\eqref{21}$, and $(\mathrm{f}1)$, for all $t \in [0,T]$
		\begin{equation}
			\begin{aligned}
				\frac{\mathrm{d}}{\mathrm{d}t} \sum_{i}\theta \left(\frac{|i|}{k}\right) |\rho_i u_i(t)|^2   \leq - \lambda \sum_i \theta\left(\frac{|i|}{k}\right) |\rho_i u_i(t)|^2 + \frac{4MC_{\theta}}{k} + \frac{1}{\lambda} \sum_{|i| \geq k} |\rho_i g_i|^2 + 2 \sum_{i}\theta \left(\frac{|i|}{k}\right) q_i(t) \rho_i^2 u_i(t) \frac{\mathrm{d}w_i}{\mathrm{d}t}.
			\end{aligned}\notag
		\end{equation}
		Since $g \in l^2_{\rho}$ and $MC_{\theta}$ is independent of $k$, we can show that, for any $\epsilon > 0$ given, there exists $K_1(\epsilon)$ such that when $k \geq K_1(\epsilon)$,
		\begin{equation}
			\begin{aligned}
				\frac{4MC_{\theta}}{k} + \frac{1}{\lambda} \sum_{|i| \geq k} |\rho_i g_i|^2 \leq \frac{\lambda}{2}\epsilon.
			\end{aligned}\notag
		\end{equation}
		Then when $k \geq K_1(\epsilon)$, we obtain
		\begin{equation}\label{22}
			\begin{aligned}
				\frac{\mathrm{d}}{\mathrm{d}t} \sum_{i}\theta \left(\frac{|i|}{k}\right) |\rho_i u_i(t)|^2   \leq - \lambda \sum_i \theta\left(\frac{|i|}{k}\right) |\rho_i u_i(t)|^2 +  \frac{\lambda}{2}\epsilon + 2 \sum_{i}\theta \left(\frac{|i|}{k}\right) q_i(t) \rho_i^2 u_i(t) \frac{\mathrm{d}w_i}{\mathrm{d}t}.
			\end{aligned}
		\end{equation}
		By integrating both sides of $\eqref{22}$ over $[0, t]$ and taking the expectation, we obtain
		\begin{equation}
			\begin{aligned}
				\sum_{i}\theta \left(\frac{|i|}{k}\right)\mathbb{E} |\rho_i u_i(t)|^2 
				&\leq e^{-\lambda t} \sum_{i}\theta \left(\frac{|i|}{k}\right) \mathbb{E}|\rho_i u_i(0)|^2 + \frac{\epsilon}{2} \quad \text{for all}~~ t \in [0,T].
			\end{aligned}\notag
		\end{equation}
		Since $u = \left( u_{i} \right)_{i \in \mathbb{Z}} \in l^2_{\rho}$, we can show that, for any $\epsilon > 0$ given, there exists $K_2(\epsilon)$ such that when $k \geq K_2(\epsilon)$,
		\begin{equation}
			\begin{aligned}
				\sum_{|i| \geq k} \mathbb{E}|\rho_i u_i(0)|^2 \leq \frac{1}{2} \epsilon.
			\end{aligned}\notag
		\end{equation}
		Finally, let $k \geq K(\epsilon) := \max \left\lbrace K_1(\epsilon), K_2(\epsilon) \right\rbrace$, we have
		\begin{equation}\label{23}
			\begin{aligned}
				\sum_{|i| \geq k}\mathbb{E} |\rho_i u_i(t)|^2 
				&\leq \epsilon \quad \text{for all}~~ t \in [0,T],
			\end{aligned}
		\end{equation}
		which is equivalent to the theorem being true.
	\end{proof}

\section{Onsager-Machlup functional for SLDSs}

For simplicity, let us define $\tilde{A} = \nu A + \lambda I$, $F(u) := - f(u) + g$ for each $u \in l^2_{\rho}$, and without confusion, we still use $A$ to represent $\tilde{A}$. Then equation $\eqref{5}$ with initial condition $\eqref{6}$ can be written in the following form:
\begin{equation}\label{24}
	\begin{aligned}
		\frac{\mathrm{d} u(t)}{\mathrm{d} t}
		= - Au(t)  + F(u) + Q(t) \frac{\mathrm{d} W(t)}{\mathrm{d} t}, \quad t \in [0,T]
	\end{aligned}
\end{equation}
with initial condition
\begin{equation}\label{25}
	\begin{aligned}
		u(0) = \left( u_{i}(0) \right)_{i \in \mathbb{Z}} \in l^2_{\rho}.
	\end{aligned}
\end{equation}
Note that the operator $ A $ is a positive self-adjoint operator (i.e. $\left( Au, u\right) > 0$ and $(Au,v)=(u,Av)$ for all $u,v \in l^2_{\rho}$), and generates a $ C_0 $-semigroup $\left\lbrace \exp(tA),~ t \geq 0\right\rbrace $, which is strongly continuous. Under a set of orthogonal bases $\left\lbrace e_i \right\rbrace_{i \in \mathbb{Z}} $ in $l^2_{\rho}$, we assume that the eigenvalues of $A$ are $\left\lbrace \alpha_i > \lambda> 0, i \in \mathbb{Z}\right\rbrace $. Furthermore, under conditions (f1) and (f2), $F \in C^2([0,1], l^2_{\rho})$ is Lipschitz continuous and 
\begin{displaymath}
	\sup\limits_{t \in [0,T]} \Vert F(u)\Vert_{\rho} \leq K(1 + \Vert u\Vert_{\rho}^{2p})\quad \text{for all}~ u \in l^2_{\rho},
\end{displaymath}
where $p$ is defined in (f2). Now, we will calculate the Onsager-Machlup functional of equation $\eqref{24}$ with initial condition $\eqref{25}$. First, we present some necessary conditions.
	\begin{itemize}
		\item[(H1)] $ Q(t)$ is a non-degenerate linear operator. In other words, $Q(t)$ is invertible, i.e. the diagonal elements $ q_i(t) $ are all invertible, for any $ t \in [0, 1] $ and $\left( q_i(t) \right)_{i \in \mathbb{Z}} \in l^2_{\rho}, q_i(t) \in C^{2}(\mathbb{R})$. Furthermore, we denote the inverse of $Q(t)$ as $Q(t)^{-1}$, and let $Q(t)$ and $Q(t)^{-1}$ be bounded for all $t \in [0,1]$.
		\item[(H2)] For any $ t \in [0, 1] $, $ F(u(t)) \in \text{Im}(Q(t)) $ almost surely, and satisfies:
		\begin{displaymath}
			\mathbb{E}\left[\int_{0}^{1} \exp\left\lbrace \frac{1}{2} \Vert Q(t)^{-1} F(u(t))\Vert^2_{\rho} \right\rbrace  dt\right] < \infty.
		\end{displaymath}
	\end{itemize}

\begin{theorem}	\label{T4.1}
	When $f$ satisfies condition $(\mathrm{f}1)$ and condition $(\mathrm{f}2)$, assume that $u(t)$ is a solution of equation $\eqref{24}$, reference path $\varphi$ is a function such that $\varphi(t) - u(0)$ belongs to Cameron-Martin $H^1$, and conditions $(H1)$,$(H2)$ hold. Then, the Onsager-Machlup functional of $u(t)$ exists and has the form
	\begin{equation}\label{26}
		\int_{0}^{1} OM(\varphi, \dot{\varphi}) \,{\rm d}t = \int_{0}^{1} \Vert Q(t)^{-1} \left({\dot{\varphi}(t) + A\varphi(t) - F(\varphi(t))} \right)\Vert^2_{\rho} \,{\rm d}t + \int_{0}^{1} { \text{Tr}_{\rho} \left(\mathcal{D}_{\varphi(t)} F \right)} \,{\rm d}t,
	\end{equation}
	where $\text{Tr}_{\rho}$ represents a weighted trace and $\mathcal{D}_{\varphi(t)} F$ represents the Fréchet derivative of $F$ at $\varphi(t)$.
\end{theorem}
\begin{proof}
	In order to apply Girsanov's transform, fix a function $ h \in L^2([0, 1]; l^2_{\rho}) $. Let $ \varphi(t) $ be the solution of the infinite dimensional equation
	\begin{equation}\label{27}
		\begin{aligned}
			\mathrm{d}\varphi(t) &= - A\varphi(t)  \mathrm{d}t + Q(t) h(t)  \mathrm{d}t, \quad t \in [0, 1]
		\end{aligned}
	\end{equation}
	with initial condition
	\begin{equation}\label{28}
		\begin{aligned}
			\varphi(0) &= u(0).
		\end{aligned}
	\end{equation}
	We will compute the Onsager–Machlup functional on $ L^2([0, 1]; l^2_{\rho}) $ at points of the form $ \varphi(t) $. Let $v_t$ be the solution for following stochastic integral equation:
	\begin{displaymath}
		v(t) = \varphi(t) + \int_{0}^{t} \left\langle {e^{-As} Q(s)}, \,{\rm d}\tilde{W}(s)\right\rangle,
	\end{displaymath}
	where $\varphi(t) \in H^1$. For convenience, denote $W^Q(t) := \int_{0}^{t} \left\langle {e^{-As} Q(s)}, \,{\rm d}\tilde{W}(s)\right\rangle$. Under the conditions of Theorem $\eqref{T4.1}$, we can verify that the Novikov condition $\mathbb{E}^{\mathbb{P}}\left( \mathrm{exp}\left\{ \int_{0}^{1} {\Vert Q(t)^{-1} \left({\dot{\varphi}(t) - F(\varphi(t))} \right)\Vert_{\rho}} \,{\rm d}t \right\} \right) < \infty$ is clearly satisfied. Girsanov theorem implies that $\tilde{W}(t) = W(t) - \int_{0}^{t} {{Q(s)^{-1}}  {F \left(v(s) \right) } } \,{\rm d}s - \int_{0}^{t} { h(s)} \,{\rm d}s$ is an infinite-dimensional Brownian motion under new probability $\tilde{\mathbb{P}}$ (see \cite{8} Theorem 10.14), which defined by $\frac{\tilde{\mathbb{P}}}{\mathbb{P}} := \mathcal{R}$ with
	\begin{displaymath}
		\mathcal{R} := {\rm exp} \left\{ {\int_{0}^{T} \big\langle  {{\left({Q(t)^{-1}} F \left(v(t) \right) - h(t) \right)} }, \,{\rm d}W_t \big\rangle_{\rho} - \frac{1}{2}\int_{0}^{T} {\Vert{ {Q \left( t \right)^{-1}}  F \left(v(t) \right) - h(t) } \Vert^2_{\rho} } \,{\rm d}t} \right\}.
	\end{displaymath}
	So we have
	\begin{equation}\label{29}
		\begin{aligned}
			&\frac{\mathbb{P}\left(\Vert u(t) -\varphi(t) \Vert_{L^2_{\rho}} \leq \epsilon\right)}{\mathbb{P}\left(\Vert W^Q \Vert_{L^2_{\rho}} \leq \epsilon\right)} 
			 = \frac{\tilde{\mathbb{P}}\left(\Vert v(t) -\varphi(t)\Vert_{L^2_{\rho}} \leq \epsilon\right)}{\mathbb{P}\left(\Vert W^Q\Vert_{L^2_{\rho}} \leq \epsilon\right)} 
			= \frac{\mathbb{E} \left( \mathcal{R}\mathbb{I}_{\Vert W^Q\Vert_{L^2_{\rho}} \leq \epsilon} \right)}{\mathbb{P}\left(\Vert W^Q\Vert_{L^2_{\rho}} \leq \epsilon\right)} = \mathbb{E}\left( \mathcal{R} \big| \Vert W^Q \Vert_{L^2_{\rho}} \leq \epsilon \right)
			\\& = \mathbb{E} \bigg( {\rm exp} \left\{ {\int_{0}^{1} \big\langle  {{\left({Q(t)^{-1}} F \left(v(t) \right) - h(t) \right)} }, \,{\rm d}W_t \big\rangle_{\rho} - \frac{1}{2}\int_{0}^{1} {\Vert{ {Q \left( t \right)^{-1}}  F \left(v(t) \right) - h(t) } \Vert^2_{\rho} } \,{\rm d}t} \right\}  \big| \Vert W^Q\Vert_{L^2_{\rho}} \leq \epsilon  \bigg) 
			\\& =  \mathbb{E} \left( {\rm exp}\left\{ { \sum_{i = 1}^{4} B_i} \right\} \big| \Vert W^Q \Vert_{L^2_{\rho}} \leq \epsilon \right)  {\rm exp}\left\{ { -\frac{1}{2} \int_{0}^{1} {\Vert Q^{-1} \left({\dot{\varphi}(t) + A\varphi(t) - F(\varphi(t))} \right)\Vert^2_{\rho}} \,{\rm d}t} \right\},
		\end{aligned}
	\end{equation}
	where
	\begin{align*}
		B_1 &= \int_{0}^{1} \left\langle  {{Q(t)^{-1}}{F(v(t))}}, \,{\rm d}W(t) \right\rangle_{\rho},
		\\ B_2 &= - \int_{0}^{1} \left\langle  {h(t) }, \,{\rm d}W(t) \right\rangle_{\rho},
		\\ B_3 &= \frac{1}{2} \int_{0}^{1} {\Vert {Q(t)^{-1}} {F(\varphi(t))} \Vert^2_{\rho} } \,{\rm d}t
		- \frac{1}{2} \int_{0}^{1} {\Vert {Q(t)^{-1}} {F(v(t)) } \Vert^2_{\rho} }  \,{\rm d}t,
		\\ B_4 &= \int_{0}^{1} { \left\langle  {{Q(t)^{-1}} \left( F(v(t)) - F(\varphi(t))\right) , h(t)} \right\rangle_{\rho} } \,{\rm d}t.
	\end{align*}
	
	For the second term $B_2$, applying Lemma $\ref{lemma 2.15}$ to $ h \in L^2([0, 1]; l^2_{\rho}) $ and Lemma $\ref{lemma 2.7}$ we have
	\begin{equation}\label{30}
		\limsup\limits_{\epsilon \to 0} \mathbb{E}\left({\rm exp}\left\{ cB_2 \right\} \big|\Vert W^Q \Vert_{L^2_{\rho}} < \epsilon \right) = 1, \quad \text{for all}~  c \in \mathbb{R}.
	\end{equation}

	For the third term $B_3$, 
	\begin{align*}
		B_3 &= \frac{1}{2} \int_{0}^{1} {\Vert {Q(t)^{-1}} {F(\varphi(t))} \Vert^2_{\rho} } - {\Vert {Q(t)^{-1}} {F(v(t)) } \Vert^2_{\rho}}  \,{\rm d}t
		\\ &\leq \frac{1}{2} \int_{0}^{1}  {\Vert {Q(t)^{-1}}\Vert^2 \Vert \left(F(\varphi(t)) - F(v(t)) \right) \Vert^2_{\rho}} \,{\rm d}t\\
		&\quad +  \int_{0}^{1} {\Vert  {Q(t)^{-1}} \Vert^2  \Vert \left(F(\varphi(t)) - F(v(t)) \right) \Vert_{\rho} \Vert F(v(t)) \Vert_{\rho}} \,{\rm d}t.
	\end{align*}
	Since $F$ is Lipschitz continuous and ${Q(t)^{-1}}$ and $F$ are bounded, we have
		\begin{displaymath}
		\begin{aligned}
			B_3 \leq C \int_{0}^{1} {\Vert W^Q(t) \Vert_{\rho}} \,{\rm d}t.
		\end{aligned}
	\end{displaymath}
	Then
	\begin{equation}\label{31}
		\limsup\limits_{\epsilon \to 0} \mathbb{E}\left({\rm exp}\left\{ cB_3 \right\} \big|\Vert W^Q \Vert_{L^2_{\rho}} < \epsilon \right) = 1 \quad \text{for all}~  c \in \mathbb{R}.
	\end{equation}
	
	For the fourth term $B_4$, applying the confition that $F$ is Lipschitz continuous and ${Q(t)^{-1}}$ is bounded, we have
	\begin{equation}
		\begin{aligned}
			B_4 &= \int_{0}^{1} { \big\langle  {Q(t)^{-1}} {(F(v(t)) - F(\varphi(t)), h(t)} \big\rangle_{\rho} } \,{\rm d}t\\
			& \leq \int_{0}^{1} { \big\langle  C W^Q(t), h(t) \big\rangle_{\rho} } \,{\rm d}t.
		\end{aligned}\notag
	\end{equation}
	On the set $\{\Vert W^Q \Vert_{L^2_{\rho}} \leq \epsilon\}$, and given that $h(t) \in l^2_{\rho}$, it is straightforward to verify that $B_4  \leq C\epsilon$, and hence
	\begin{equation}\label{32}
		\limsup\limits_{\epsilon \to 0} \mathbb{E}\left({\rm exp}\left\{ cB_4 \right\} \big|\Vert W^Q \Vert_{L^2_{\rho}} < \epsilon \right) = 1 \quad \text{for all}~  c \in \mathbb{R}.
	\end{equation}
	
	For the first term $B_1$, applying Taylor's expansion, we have
	\begin{displaymath}
		\begin{aligned}
			F(v(t)) &= F(\varphi(t) + W^Q(t)) = F(\varphi(t)) + \mathcal{D}_{\varphi(t)} (F) W^Q(t) + R(t),
		\end{aligned}
	\end{displaymath}
	where $\mathcal{D}_{\varphi(t)} (F)$ represents the Fréchet derivative of $F$ at $\varphi(t)$ and $R(t)$ represents the first-order remainder term in the Taylor's expansion. According to $F \in C^2([0,1],l^2_{\rho})$, when $\Vert W^Q(t) \Vert_{L^2_{\rho}} \leq \epsilon$, we have
	\begin{displaymath}
		 {\Vert R(t) \Vert_{L^2_{\rho}}}  \leq C\epsilon^2.
	\end{displaymath}
	Hence, $B_1$ can be written
	\begin{displaymath}
		\begin{aligned}
			B_1 &= \int_{0}^{1} \big\langle {{Q(t)^{-1}}{F(v(t))}}, \,{\rm d}W(t) \big\rangle_{\rho} = \int_{0}^{1} \big\langle {{Q(t)^{-1}} {\left(F(\varphi(t)) + \mathcal{D}_{\varphi(t)} F (W^Q(t)) + R(t) \right)}}, \,{\rm d}W(t) \big\rangle_{\rho}
			\\ &:= B_{11} + B_{12} + B_{13}.
		\end{aligned}
	\end{displaymath}

	The term $B_{11}$ has the same expression as $B_2$:
	\begin{displaymath}
		\begin{aligned}
			B_{11} = \int_{0}^{1} \big\langle {{Q(t)^{-1}} {F(\varphi(t))}}, \,{\rm d}W(t) \big\rangle_{\rho}.
		\end{aligned}
	\end{displaymath}
	Due to $F \in C^2_b \left( l^2_{\rho}, l^2_{\rho} \right)$, we can show that ${Q(t)^{-1}} {F(\varphi(t))}  \in L^2([0,1], l^2_{\rho})$. Using the same method as item $B_2$, we have
	\begin{equation}\label{33}
		\limsup\limits_{\epsilon \to 0} \mathbb{E}\left({\rm exp}\left\{ cB_{11} \right\} \big|\Vert W^Q \Vert_{L^2_{\rho}} < \epsilon \right) = 1\quad \text{for all}~c \in \mathbb{R}.
	\end{equation}

	The term $B_{12}$ is a double stochastic integral with respect to $W$. Here, we use the techniques of the Karhunen-Loève expansion and Skorohod integral to transform $W^Q(t)$ into a form of independent $\mathcal{N}(0, 1)$ random variables, and ultimately derive the final result from the Lemma $\ref{lmma 2.9}$. First, based on the definition of $W^Q(t)$ and Theorem \ref{theorem 2.6}, we present the Karhunen-Loève expansion of $\rho W^Q(t) $
	\begin{displaymath}
		\rho W^Q(t) = \sum_{i=1}^{\infty} \left( \int_{0}^{t} {\rho}_i q_i(s) e^{-\alpha_i(t-s)} \, dW_i(s) \right) e_i
		= \sum_{i=1}^{\infty} \sum_{j=1}^{\infty} \lambda_{j,i} x_{j,i} (l_{j,i} \otimes e_i)(t),
	\end{displaymath}
	where $l_{j,i}(t)$ is an orthonormal basis of $L^2([0,1])$, $(l_{j,i} \otimes e_i)(t) := l_{j,i}(t)e_i$ and
	\begin{equation}
		\begin{aligned}
			&\lambda_{j,i}^2 = \mathrm{Var} \left[ \int_{0}^{1} \left( {\int_{0}^{t} {\rho}_i q_i(s) e^{-\alpha_i(t-s)} \mathrm{d}W_i(s)}\right) l_{j,i} (t) \mathrm{d}t \right],\\
			&x_{j,i}(\omega) = \frac{1}{\lambda_i} \int_{0}^{1} \left( {\int_{0}^{t} {\rho}_i q_i(s) e^{-\alpha_i(t-s)} \mathrm{d}W_i(s)}\right) l_{j,i}(t) \mathrm{d}t.
		\end{aligned}\notag
	\end{equation}
	
	Note that $\{l_{j,i} \otimes e_i, j \geq 1, i \geq 1\}$ is an orthonormal basis of $L^2([0, 1]; l^2_{\rho})$ such that for any $i, j \geq 1$: $\text{Cov}((W^Q, l_{j,i} \otimes e_i)_{L^2([0,1]; l^{\infty})}) = \lambda_{j,i}^2$.
	Thus,
	\begin{displaymath}
		\Vert W^Q \Vert^2_{\rho} = \sum_{j=1}^{\infty} \sum_{k=1}^{\infty} \lambda_{j,i}^2 x_{j,i}^2,
	\end{displaymath}
	due to $\alpha_i > 0$ and $\left\lbrace q_i(t) \right\rbrace_{i \in \mathbb{Z}} >0 \in L^2([0, 1]; l^2_{\rho})$ for all $i \in \mathbb{Z}$, we can show that $\sum_{i,j=1}^{\infty} \lambda_{j,i}^2 < +\infty$. 
	
	Consider also
	\begin{displaymath}
		h_{j,i}(s) := \frac{1}{\lambda_{j,i}} \int_{s}^{1} {\rho}_i q_i(t) e^{-\alpha_i(t-s)} l_{j,i}(t) \mathrm{d}t
	\end{displaymath}
	for $ i \geq 1, j \geq 1 $. Then, for any $ i \geq 1 $, $\{h_{j,i}, j \geq 1\}$ is an orthonormal basis of $L^2([0, 1])$. Notice that
	\begin{displaymath}
		\begin{aligned}
			\langle h_{j1,i}, h_{j2,i} \rangle 
			&=\frac{{\rho}_i^2}{\lambda_{j_1,i} \lambda_{j_2,i}} \int_{0}^{1} \left( \int_{s}^{1} q_i(t_1) e^{-\alpha_i(t_1-s)} l_{j_1,i}(t_1) \mathrm{d}t_1\right)  \left( \int_{s}^{1}  q_i(t_2) e^{-\alpha_i(t_2-s)} l_{j_2,i}(t_2) \mathrm{d}t_2\right)  \mathrm{d}s \\ 
			&=\frac{1}{\lambda_{j1,i} \lambda_{j2,i}} \int_{0}^{1} \int_{0}^{1} k_i(t, s) l_{j_1,i}(t) l_{j_2,i}(s) \mathrm{d}t~ \mathrm{d}s \\
			&= \langle l_{j_1,i}, l_{j_2,i} \rangle
		\end{aligned}
	\end{displaymath}
	for any $j_1, j_2 > 1$, where $k_i$ denotes the covariance function defined as follows:
	\begin{displaymath}
		k_i(t, s) = \int_{0}^{t \wedge s}  {\rho}_i^2 q_i(t)e^{-\alpha_i (t-u)} q_i(s)e^{-\alpha_i (s-u)} \mathrm{d}u.
	\end{displaymath}
	Therefore, to prove that $\{h_{j,i}, j \geq 1\}$ is a basis, one only needs to show that if $ h \in L^2([0, 1],l^2_{\rho}) $, for all $ j \geq 1, \left\langle  h_{j,i}, h \right\rangle = 0$, it follows that $ h $ must be zero. This directly follows from the fact that if 
	\begin{displaymath}
		0 = \left\langle  h_{j,i}, h \right\rangle   = \frac{1}{\lambda_{j,i}} \int_{0}^{1} \left( \int_{s}^{1} {\rho}_i q_i(t) e^{-\alpha_j(t-s)} l_{j,i}(t) \mathrm{d}t \right) h(s) \mathrm{d}s
		= \frac{1}{\lambda_{j,i}} \left\langle  l_{j,i}, \varphi^h \right\rangle  
	\end{displaymath}
	for all $ j \geq 1 $, then due to $\varphi^h = 0$ with $\varphi^h(t) = \int_{0}^{t} {\rho}_i q_i(t) e^{-\alpha_j(t-s)}h(s) \mathrm{d}s$, we have $ h \equiv 0 $.
	
	Furthermore,
	\begin{displaymath}
		\begin{aligned}
			x_{j,i} &= \frac{1}{\lambda_{j,i}} \int_{0}^{1} \left( \int_{0}^{t} {\rho}_i q_i(t) e^{-\alpha_j(t-s)} \mathrm{d}W_i(s) \right) l_{j,i}(t) \mathrm{d}t\\
			&= \int_{0}^{1} \left( \frac{1}{\lambda_{j,i}} \int_{s}^{1} {\rho}_i q_i(t)e^{-\alpha_j(t-s)} l_{j,i}(t) \mathrm{d}t \right)  \mathrm{d}W_i(s)\\
			&= I_i(h_{j,i}),
		\end{aligned}
	\end{displaymath}
	where $ I_i(h_{j,i}) = \int_{0}^{1} h_{j,i}(s) \mathrm{d}W_i(s) $.
	
	For clarity, we introduce the following notation. Define $\hat{P} = \rho Q^{-1} F$, which is assumed to be $C_b^2$ in $x$, uniformly in $s \in [0, 1]$. Let $P$ and $R$ be linear operators defined on $L^2([0, 1]; l^2_{\rho})$. For any $f \in L^2([0, 1]; l^2_{\rho})$ and $e_i \in l^2_{\rho}$, the operator $P$ is given by
	\begin{displaymath}
		(P(f \otimes e_i))(s) := f(s)(\mathcal{D}_{\varphi_{(s)}} \hat{P})(e_i).
	\end{displaymath} 
	The operator $R$ is defined as
	\begin{displaymath}
		R(f \otimes e_i) = R_i(f) \otimes e_i,
	\end{displaymath} 
	where
	\begin{displaymath}
		(R_i f)(s) := \int_{s}^{1} \rho_i q_i(t) e^{-\alpha_i(t-s)} f(t) \, \mathrm{d}t.
	\end{displaymath} 
	
	Note that for any $i \geq 1, j \geq 1, h_{j,i} = \frac{1}{\lambda_{j,i}} R_i l_{j,i}$ and $ Q (t) $ is a linear operator about time $ t $ and does not depend on $ x $. Then using the Karhunen-Loève expansion, we have
	\begin{equation}\label{34}
		\begin{aligned}
			B_{12} &= \int_{0}^{1} \left\langle  {{Q(t)^{-1}} \mathcal{D}_{\varphi(t)} F (W^Q(t)) }, \,{\rm d}W(t) \right\rangle_{\rho} \\
			&= \int_{0}^{1} \left\langle \mathcal{D}_{\varphi(t)} \hat{P} (\rho W^Q(t)) ,\mathrm{d}W(t) \right\rangle\\
			&= \sum_{i,j,k=1}^{\infty} \lambda_{j,i} \int_{0}^{1} x_{j,i} \left\langle  P (l_{j,i}(t) \otimes e_i), e_k \right\rangle \mathrm{d}W_k(t).
		\end{aligned}
	\end{equation}
	It is noteworthy that the random variable $ x_{j,i}$ is measurable with respect to $\mathcal{F}_1$. Hence, the introduction of Skorohod integral is necessitated to handle the anticipating stochastic integrals that arise in the above expression. We transition from Itô integrals to Skorohod integrals, capitalizing on the property that they concur on the set \( L^2_a \) of square-integrable adapted processes. For a detailed discussion on Skorohod integrals, refer to \cite{9}. When $k = i$,
	\begin{displaymath}
		\begin{aligned}
			& \quad \int_{0}^{1} x_{j,i} \left\langle  P (l_{j,i}(t) \otimes e_i), e_k \right\rangle \mathrm{d}W_k(t) \\& = \int_{0}^{1} x_{j,i} \left\langle  P (l_{j,i}(t) \otimes e_i), e_i \right\rangle \mathrm{d}W_i(t) \\& = x_{j,i} \sum_{n=1}^{\infty} \left\langle P (l_{j,i}(t) \otimes e_i), h_{n,i} \otimes e_i \right\rangle I_i(h_{n,i}) - \left\langle h_{j,i} \otimes e_i, P (l_{j,i}(t) \otimes e_i) \right\rangle,
		\end{aligned}
	\end{displaymath}
	and when $k \neq i$,
	\begin{displaymath}
		\int_0^1 x_{j,i} \left\langle  P (l_{j,i}(t) \otimes e_i), e_k \right\rangle \mathrm{d}W_k(t) = x_{j,i} \sum_{n=1}^{\infty} \left\langle  P (l_{j,i}(t) \otimes e_i), h_{n,k} \otimes e_k \right\rangle I_k(h_{n,k}).
	\end{displaymath}
	Using the fact that $h_{n,k} = \frac{1}{\lambda_{n,k}} R_k l_{n,k}$ and $I_k(h_{n,k}) = x_{n,k}$, we can write \eqref{34} in the following way:
	\begin{displaymath}
		\begin{aligned}
			B_{12} &= \sum_{(j,i) \neq (n,k)} \frac{\lambda_{j,i}}{\lambda_{n,k}} x_{j,i} 	x_{n,k} \left\langle  P (l_{j,i}(t) \otimes e_i), R_k l_{n,k} \otimes e_k \right\rangle\\
			&\quad + \sum_{j,i} (x_{j,i}^2 - 1) \left\langle  P (l_{j,i}(t) \otimes e_i), R_i l_{j,i} \otimes e_i \right\rangle.
		\end{aligned}
	\end{displaymath}
	Define now the operator $T : l^2_{\mathbb{Z}^2} \rightarrow l^2_{\mathbb{Z}^2}$ by
	\begin{displaymath}
		T_{(j,i),(n,k)} = \frac{\lambda_{j,i}}{\lambda_{n,k}}  \left\langle  P (l_{j,i}(t) \otimes e_i), R_k l_{n,k} \otimes e_k \right\rangle, \quad (j, i), (n, k) \in \mathbb{Z}^2.
	\end{displaymath}
	Then
	\begin{displaymath}
		\begin{aligned}
			B_{12} &= \sum_{(j,i) \neq (n,k)} T_{(j,i),(n,k)} x_{j,i} x_{n,k} + \sum_{j,i} T_{(j,i),(j,i)} (x_{j,i}^2 - 1)\\
			&= \sum_{j,i,n,k} T_{(j,i),(n,k)} x_{j,i} x_{n,k} + \sum_{j,i} T_{(j,i),(j,i)} (x_{j,i}^2 - 2).
		\end{aligned}
	\end{displaymath}
	Due to
	\begin{displaymath}
		\begin{aligned}
			R^*R\left( l_{j,i} \otimes e_i\right) &= \sum_{n=1}^{\infty} ((R^* R)_i l_{j,i}, l_{m,i}) (l_{m,i} \otimes e_i)\\
			&= \sum_{n=1}^{\infty} (R_i l_{j,i}, R_i l_{m,i}) (l_{m,i} 	\otimes e_i) \\
			& = \lambda_{j,i}^2\left( l_{j,i} \otimes e_i\right),
		\end{aligned}
	\end{displaymath}
	and
	\begin{displaymath}
		\begin{aligned}
			\frac{1}{\lambda_{j,i}}R\left( l_{j,i} \otimes e_i\right) = (h_{j,i} \otimes e_i),
		\end{aligned}
	\end{displaymath}
	we have
	\begin{displaymath}
		\begin{aligned}
			T_{(j,i),(n,k)} &= \langle  \lambda_{j,i} P\left( l_{j,i} \otimes e_i\right) , \frac{1}{\lambda_{n,k}} R\left( l_{n,k} \otimes e_k\right)  \rangle\\
			&= \langle  \frac{1}{\lambda_{j,i}} PR^*R\left( l_{j,i} \otimes e_i\right) , \frac{1}{\lambda_{n,k}} R\left( l_{n,k} \otimes e_k\right)  \rangle\\
			&= \langle P R^*(h_{j,i} \otimes e_i), h_{n,k} \otimes e_k \rangle.
		\end{aligned}
	\end{displaymath}
	According to Lemma \ref{lmma 2.9}, we define the self adjoint operator $\tilde{T}=\frac{1}{2} \left(P R^* + \left( P R^*\right)^*  \right) $. When $\tilde{T}$ is a trace operator, we have
	\begin{displaymath}
		\limsup\limits_{\epsilon \to 0} \mathbb{E}\left(\text{exp} \left\lbrace c\sum_{j,i,n,k} T_{(j,i),(n,k)} x_{j,i} x_{n,k}\right\rbrace  \big|\Vert W^Q \Vert_{L^2_{\rho}} < \epsilon \right) = 1 \quad \text{for all}~ c \in \mathbb{R}.
	\end{displaymath}
	Therefore,
	\begin{equation}
		\begin{aligned}
			&\quad \limsup\limits_{\epsilon \to 0} \mathbb{E}\left({\rm exp}\left\{ B_{12} \right\} \big|\Vert W^Q \Vert_{L^2_{\rho}} < \epsilon \right) \\
			& = \limsup\limits_{\epsilon \to 0} \mathbb{E}\left({\rm exp}\left\{ \sum_{j,i} T_{(j,i),(j,i)} (x_{j,i}^2 - 2) \right\} \big|\Vert W^Q \Vert_{L^2_{\rho}} < \epsilon \right)\\
			& = {\rm exp}\left\{ - { Tr\left({\tilde{T}} \right)}  \right\}.
		\end{aligned}\notag
	\end{equation}

	Given that $R(f \otimes e_i) = R_i(f) \otimes e_i$ with
	\begin{displaymath}
		(R_i f)(s) := \int_{s}^{1} {\rho}_i q_i(t) e^{-\alpha_i(t-s)} f(t) \mathrm{d}t.
	\end{displaymath}
	We have $R^*(f \otimes e_i) = R^*_i(f) \otimes e_i$ with
	\begin{displaymath}
		(R_i^* f)(s) = \rho_i q_i(s) \int_0^s e^{-\alpha_i(s-t)} f(t) \, \mathrm{d}t.
	\end{displaymath}
	Then,
	\begin{displaymath}
		\begin{aligned}
			(PR^*)(f \otimes e_i)(s) &= \sum_{n=1}^{\infty} \left( \rho_i^2 \int_0^s e^{-\alpha_i(s-t)} f(t) \, \mathrm{d}t \left\langle \mathcal{D}_{\varphi_{(s)}} F(e_i), e_n \right\rangle(s) \right)  e_n
			\\(PR^*)^*(f \otimes e_i)(s) &= \sum_{n=1}^{\infty} \left( \rho_i^2 \int_s^1 e^{-\alpha_i(t-s)} f(t) \left\langle \mathcal{D}_{\varphi_{(t)}} F(e_i), e_n \right\rangle (t) \, \mathrm{d}t \right) e_n
		\end{aligned}\notag
	\end{displaymath}
	Thus, we have
	\begin{align*}
		&\quad \sum_{i,j=1}^{\infty}  \frac{1}{2} \left( P R^* + (P R^*)^* \right) (f \otimes e_i)(s)\\
		&= \frac{\rho_i^2}{2} \sum_{i,j=1}^{\infty} \left(  \int_0^1 e^{-\alpha_i(s-t)} f(t) \left\langle \mathcal{D}_{\varphi_{(s)}} F(e_i), e_i \right\rangle(s) 1_{[0,s]}
		+ e^{-\alpha_i(t-s)} f(t) \left\langle \mathcal{D}_{\varphi_{(s)}} F(e_i), e_i \right\rangle (t) 1_{[s,1]} \, \mathrm{d}t \right) e_i\\
		&= \frac{\rho_i^2}{2} \sum_{i,j=1}^{\infty} \left(  \int_0^1 e^{-\alpha_i|s-t|}  \left\langle \mathcal{D}_{\varphi_{(t \vee s)}} F(e_i), e_i \right\rangle(t \vee s)  f(t) \, \mathrm{d}t \right) e_i\\
		&=\frac{\rho_i^2}{2} \sum_{i,j=1}^{\infty} \left(  \int_0^1 \bar{K}(s,t)  f(t) \, \mathrm{d}t \right) e_i.
	\end{align*}
	Last, we have
	\begin{align*}
		&\text{Tr} \left( \frac{1}{2} \left( P R^* + (P R^*)^* \right) \right) \\
		&= \sum_{i,j=1}^{\infty} \left\langle \frac{1}{2} \left( P R^* + (P R^*)^* \right) ({l}_{j,i} \otimes e_i), {l}_{j,i} \otimes e_i \right\rangle \\
		&=\frac{\rho_i^2}{2} \sum_{i,j=1}^{\infty} \left(  \int_0^1 \bar{K}(t,t)  \left\langle {l}_{j,i}, {l}_{j,i}\right\rangle  \, \mathrm{d}t \right)\\
		&= \frac{\rho_i^2}{2} \sum_{i,j=1}^{\infty} \left(  \int_0^1   \left\langle \mathcal{D}_{\varphi_{(t)}} F({l}_{j,i} \otimes e_i), {l}_{j,i} \otimes e_i \right\rangle(t)  \, \mathrm{d}t \right)\\
		&= \frac{1}{2} \int_0^1 \text{Tr}(\rho^2 \mathcal{D}_{\varphi(t)} F) \mathrm{d}t.
	\end{align*}
	We define $\text{Tr}_{\rho} (\mathcal{D}{\varphi(t)} F)$ as $\text{Tr}_{\rho} (\mathcal{D}{\varphi(t)} F) := \text{Tr}(\rho^2 \mathcal{D}_{\varphi(t)} F)$. Therefore,
	\begin{equation}\label{35}
		\begin{aligned}
			\quad \limsup\limits_{\epsilon \to 0} \mathbb{E}\left({\rm exp}\left\{ B_{12} \right\} \big|\Vert W^Q \Vert_{L^2_{\rho}} < \epsilon \right) 
			= {\rm exp}\left\{ -\frac{1}{2} \int_{0}^{1} { \text{Tr}_{\rho} \left({\mathcal{D}_{\varphi(t)}F} \right)} \,{\rm d}t  \right\}
		\end{aligned}
	\end{equation}
	for all $c \in \mathbb{R}$.
	
	Finally, we study the behaviour of the term $B_{13}$. For any $c \in \mathbb{R}$ and $\delta > 0$, we have
	\begin{equation}\label{36}
		\begin{aligned}
			\mathbb{E} \left( {\rm exp} \left\{ cB_{13} \right\} \big| \Vert W^Q \Vert_{L^2_{\rho}} \leq \epsilon \right) 
			&= \int_{0}^{\infty} {e^x \mathbb{P}\left( \left| c\int_{0}^{1} \left\langle {Q(t)^{-1}} {R(t)}, \,{\rm d}W(t) \right\rangle_{\rho} \right| > x \big| \Vert W^Q \Vert_{L^2_{\rho}} \leq \epsilon \right)} \,{\rm d}x
			\\ &\leq \int_{\delta}^{\infty} {e^x \mathbb{P}\left( \left| c\int_{0}^{1} \left\langle  {Q(t)^{-1}} {R(t)}, \,{\rm d}W(t) \right\rangle_{\rho} \right| > x \big| \Vert W^Q \Vert_{L^2_{\rho}} \leq \epsilon \right)} \,{\rm d}x 
			\\ &\quad + e^{\delta} \mathbb{P}\left( \left| c\int_{0}^{1} \left\langle  {Q(t)^{-1}} {R(t)}, \,{\rm d}W(t) \right\rangle_{\rho} \right| > \delta \big| \Vert W^Q \Vert_{L^2_{\rho}} \leq \epsilon \right).
		\end{aligned}
	\end{equation}
	Define the martingale $M_t = c\int_{0}^{t} \left\langle  {Q(s)^{-1}} {R(s)}, \,{\rm d}W(s) \right\rangle_{\rho}$. We have the estimate about its quadratic variation
	\begin{displaymath}
		\langle M_t \rangle = c^2\int_{0}^{t} {\Vert {Q(s)^{-1}}  R(s) \Vert^2_{\rho}} \,{\rm d}s \leq C \epsilon^4
	\end{displaymath}
	for some $C > 0$. Using the exponential inequality for martingales, we have
	\begin{displaymath}
		\mathbb{P}\left( \left| c\int_{0}^{1} \left\langle  {Q(t)^{-1}} {R_t}, \,{\rm d}W_t \right\rangle_{\rho} \right| > x, \Vert W^Q \Vert_{L^2_{\rho}} \leq \epsilon \right) \leq {\rm exp}\left\{ -\frac{x^2}{2C\epsilon^4} \right\}.
	\end{displaymath}
	Below we estimate $\mathbb{P}\left( \Vert W^Q \Vert_{L^2_{\rho}} \leq \epsilon \right)$,
	\begin{displaymath}
		\begin{aligned}
			{\rho} W^Q &=  \sum_{i=1}^{\infty} \left( \int_{0}^{t} {\rho}_i q_i(s) e^{-\alpha_i(t-s)} \, dW_i(s) \right) e_i\\
			& \leq \sum_{i=1}^{\infty}\max_{t \in [0,1]}\left( {\rho}_i q_i(s) \right) \int_{0}^{t}  e^{-\lambda(t-s)} \, dW_i(s)\\
			& \leq  M  \int_{0}^{t}  e^{-\lambda(t-s)} \, dW_i(s) := M  Y(t),
		\end{aligned}
	\end{displaymath}
	where $M$ is a constant and $Y(t)$ is a centered mean-square continuous stochastic process. So $\mathbb{P}\left( \Vert W^Q \Vert_{L^2_{\rho}} \leq \epsilon \right) \geq \mathbb{P}\left( \Vert Y(t) \Vert_{L^2} \leq \frac{\epsilon}{M} \right)$. Using Karhunen-Loève expansion, we have
	\begin{displaymath}
		K(t, s) = \int_0^{t \wedge s} e^{-\lambda(t-u)} e^{-\lambda(s-u)} \, du = \frac{1}{2\lambda} \left( e^{-\lambda(t \vee s - t \wedge s)} - e^{-\lambda(t+s)} \right),\quad s,t \in [0,1].
	\end{displaymath}
	
	To determine the eigenvalues and the eigenvectors of the symmetric operator associated with the kernel $K$ in $L^2([0, 1])$, we need to solve the equation
	\begin{displaymath}
		\int_0^1 K(t, s) h(t) \, dt = \mu h(s), \quad 0 \leq s \leq 1,
	\end{displaymath}
	Differentiating twice, we find that $h$ satisfies
	\begin{equation}\label{37}
		(\lambda^2 \mu - 1) h(s) = \mu h''(s), \quad 0 \leq s \leq 1,
	\end{equation}
	with initial conditions $h(0) = 0$ and $\lambda h(1) = -h'(1)$. Observe that $\eqref{37}$ clearly implies that $\mu \neq 0$.
	
	Set $\beta = \frac{\mu}{\lambda^2 \mu - 1}$ and assume $\beta$ is well-defined and strictly negative. Then the solution of $\eqref{37}$  takes the form
	\begin{displaymath}
		h(s) = c_1 \sin \left( \frac{s}{\sqrt{|\beta|}} \right) + c_2 \cos \left( \frac{s}{\sqrt{|\beta|}} \right).
	\end{displaymath}
	The condition $h(0) = 0$ implies $c_2 = 0$, and $\lambda h(1) = -h'(1)$ yields
	\begin{displaymath}
		\tan \left( \frac{1}{\sqrt{|\beta|}} \right) = -\frac{1}{\lambda} \frac{1}{\sqrt{|\beta|}}.
	\end{displaymath}
	Set $\gamma = |\beta|^{-1/2}$. The relation $\beta = \frac{\mu}{\lambda^2 \mu - 1}$ implies that the eigenvalues of the operator $K$ form a family $\{\mu_i; i \geq 1\}$, where $\mu_i = \frac{1}{\lambda^2 + \gamma_i^2}$ and $\gamma_i$ is the solution of the equation $\tan(\gamma) = -\frac{\gamma}{\lambda}$ in the interval $[ \frac{(2i-1)\pi}{2}, \frac{(2i+1)\pi}{2} ]$. And the orthonormalized eigenfunctions are of the form $g_i(s) = A_i \mathrm{sin}(\gamma_is), i \geq 1$, where $\Vert A_i \Vert \leq 2$.

	Therefore, it is straightforward to verify that the conditions required by Lemma $\ref{lemma 2.13}$ are satisfied. Consequently, we can obtain the desired result by applying the Lemma $\ref{lemma 2.13}$. However, for convenience, we use the approximation $\mu_t = \frac{1}{\lambda^2 + \gamma_t^2} \sim t^{-2}$ to approximate the estimate, where $\gamma_t$ is the solution of the equation $\tan(\gamma) = -\frac{\gamma}{\lambda}$ in the interval $\left[ \frac{(2t-1)\pi}{2}, \frac{(2t+1)\pi}{2} \right]$. By applying Corollary $\ref{corollary 2.11}$, we obtain the following result:
	\begin{equation}\label{38}
		\begin{aligned}
			\mathbb{P}\left( \left| c\int_{0}^{1} \left\langle  {Q(t)^{-1}}  {R_t}, \,{\rm d}W_t \right\rangle_{\rho} \right| > x \big| \Vert W^Q \Vert_{L^2_{\rho}} \leq \epsilon \right) &= \frac{\mathbb{P}\left( \left| c\int_{0}^{1} \left\langle  {Q(t)^{-1}}  {R_t}, \,{\rm d}W_t \right\rangle_{\rho} \right| > x , \Vert W^Q \Vert_{L^2_{\rho}} \leq \epsilon \right)}{\mathbb{P}\left( \Vert W^Q \Vert_{L^2_{\rho}} \leq \epsilon \right)} \\
			&\leq \frac{C_1}{ \epsilon^{2}} {\rm exp}\left\{ -\frac{x^2}{2C\epsilon^4} + \frac{8M^{\frac{2}{3}}}{3\epsilon^{\frac{2}{3}}}  \right\},
		\end{aligned}
	\end{equation}
	where $C$ and $C_1$ is a positive constant.
	
	Taking the limit in $\eqref{38}$ and substituting it into $\eqref{36}$, we have
	\begin{equation}\label{39}
		\limsup\limits_{\epsilon \to 0} \mathbb{E}\left({\rm exp}\left\{ cB_{13} \right\} \big|\Vert W^Q \Vert_{L^2_{\rho}} < \epsilon \right) = 1
	\end{equation}
	for all $c \in \mathbb{R}$, as $\epsilon \to 0$ and $\delta \to 0$.
	
	Finally, by Lemma $\ref{lemma 2.7}$ and inequalities  $\eqref{29}$ - $\eqref{33}$, $\eqref{35}$ and $\eqref{39}$, we have
	\begin{displaymath}
		\begin{aligned}
			\lim\limits_{\epsilon \to 0} \frac{\mathbb{P}\left(\Vert X_t -\varphi_t \Vert_{L^2_{\rho}} \leq \epsilon\right)}{\mathbb{P}\left(\Vert W^Q \Vert_{L^2_{\rho}} \leq \epsilon\right)} = {\rm exp} \left\{  -\frac{1}{2} \int_{0}^{1} \Vert Q(t)^{-1} \left({\dot{\varphi}(t) + A\varphi(t) - F(\varphi(t))} \right) \Vert^2_{\rho} \,{\rm d}t - \frac{1}{2} \int_{0}^{1} { \text{Tr}_{\rho} \left(\mathcal{D}_{\varphi(t)}F \right)} \,{\rm d}t \right\}.	
		\end{aligned}\notag
	\end{displaymath}
	Consequently,
	\begin{displaymath}
		\int_{0}^{1} OM(\varphi, \dot{\varphi}) \,{\rm d}t =  \int_{0}^{1} \Vert Q(t)^{-1} \left({\dot{\varphi}(t) + A\varphi(t) - F(\varphi(t))} \right) \Vert^2_{\rho} \,{\rm d}t + \int_{0}^{1} { \text{Tr}_{\rho} \left(\mathcal{D}_{\varphi(t)}F \right)} \,{\rm d}t.
	\end{displaymath}
	\end{proof}

\section{Numerical experiments}
The study of stochastic lattice equations in this paper holds significant importance in biomathematics, particularly in modeling the spatial spread of diseases within populations and evaluating the effectiveness of various preventive measures. The Onsager-Machlup functional, as a tool for describing the most probable path of a stochastic process, offers several benefits: determining the most likely evolution path of disease spread by calculating key parameters such as infection and recovery rates; assessing the effectiveness of different preventive strategies; and quantifying the impact of environmental noise and stochastic perturbations on the spread of disease. These insights can help design more effective preventive measures.
\begin{example}
	Let's establish a disease transmission model:
	\begin{equation}\label{40}
		\frac{\mathrm{d} u_{i}(t)}{\mathrm{d} t} = 0.1 \left( u_{i-1} - 2u_{i} + u_{i+1} \right) - 0.4 u_{i} - 0.1 u_{i}^3 +  0.01 \left( 31 - t + \frac{1}{|i| + 1} \right)  \frac{\mathrm{d} w_{i}(t)}{\mathrm{d} t}, \quad i \in \mathbb{Z},
	\end{equation}
	where the diffusion coefficient $\nu = 0.1$ represents the rate of diffusion of infected individuals between neighboring locations; the decay coefficient $\lambda = 0.4$ denotes the rate of recovery or death; the nonlinear interaction term $f(u_i) = 0.1 u_i^3$ models contact transmission among infected individuals; the external intervention $g = 0$ accounts for preventive measures; the noise intensity $q_i(t) = 0.01(31 - t + \frac{1}{|i| + 1})$ characterizes environmental stochastic perturbations; and the initial conditions are given by $u_i(0) = 0.6 \exp \left( -\frac{i^2}{2\sigma^2} \right)$, where $\sigma = 8$.
	
	First, it is easy to verify that $f (u)= 0.1 u^3$ satisfies the conditions $(\mathrm{f}1)$ and $(\mathrm{f}2)$. Then, according to the definition of $q_{i}(t)$, we can verify that $Q(t)$ is a non-degenerate linear operator, $Q(t)$ and $Q(t)^{-1}$ are bounded for all $t \in [0,30]$. Finally according to Theorem $\ref{T3.1}$, $u(t)$ is the only bounded solution of equation $\eqref{40}$, so it can be proven that
	\begin{displaymath}
		\mathbb{E}\left[\int_{0}^{30} \exp\left\lbrace \frac{1}{2} \Vert Q(t)^{-1} F(u(t))\Vert^2 \right\rbrace  dt\right] < \infty.
	\end{displaymath}
	In summary, the condition of Theorem $\ref{T4.1}$ holds, and we can obtain the Onsager-Machlup functional as follows:
	\begin{displaymath}
		\begin{aligned}
			OM(\varphi, \dot{\varphi}) &=   \left| 	Q(t)^{-1} 	\left({\dot{\varphi}(t) + A\varphi(t) - F(\varphi(t))} \right) \right|^2  + { Tr\left(\mathcal{D}_{\varphi(t)}F \right)} \\
			&= \sum_i  \left| \frac{1}{0.01(31 - t + \frac{1}{|i| + 1})} \left( \dot{\varphi}_i(t) - 0.1 (\varphi_{i-1}(t) - 2\varphi_{i}(t) + \varphi_{i+1}(t)) + 0.4 \varphi_i(t) + 0.1 \varphi_i(t)^3 \right) \right|^2 
			\\ & \quad - 0.3 \sum_i  \varphi_i(t)^2  
			.
		\end{aligned}
	\end{displaymath}
	To solve the Onsager-Machlup functional using the Euler-Lagrange equations, we need to calculate the derivatives of each term, particularly focusing on the derivatives with respect to $\varphi_i(t)$ and $\dot{\varphi}_i(t)$.\\
	Derivative with respect to $\dot{\varphi}_i(t)$
	\begin{displaymath}
		\frac{\partial OM}{\partial \dot{\varphi}_i(t)} = \frac{2 \left( \dot{\varphi}_i(t) - 0.1 (\varphi_{i-1}(t) - 2\varphi_{i}(t) + \varphi_{i+1}(t)) + 0.4 \varphi_i(t) + 0.1 \varphi_i(t)^3 \right)}{ 0.0001\left(31 - t + \frac{1}{|i| + 1} \right)^2}.
	\end{displaymath}
	Further, we have
	\begin{displaymath}
		\begin{aligned}
			\frac{d}{dt} \left( \frac{\partial OM}{\partial \dot{\varphi}_i(t)} \right) &=\frac{ 2\left( \ddot{\varphi}_i(t) - 0.1 (\dot{\varphi}_{i-1}(t) - 2\dot{\varphi}_{i}(t) + \dot{\varphi}_{i+1}(t)) + 0.4 \dot{\varphi}_i(t) + 0.3 \varphi_i(t)^2 \dot{\varphi}_i(t)\right) }{ 0.0001\left(31 - t + \frac{1}{|i| + 1} \right)^2}    \\
			&\quad + \frac{0.04 \left( \dot{\varphi}_i(t) - 0.1 (\varphi_{i-1}(t) - 2\varphi_{i}(t) + \varphi_{i+1}(t)) + 0.4 \varphi_i(t) + 0.1 \varphi_i(t)^3 \right) }{ 0.000001\left(31 - t + \frac{1}{|i| + 1} \right)^3}.
		\end{aligned}
	\end{displaymath}
	Derivative with respect to $\varphi_i(t)$
	\begin{displaymath}
		\begin{aligned}
			\frac{\partial OM}{\partial \varphi_i(t)} &= \frac{2 \left( \dot{\varphi}_i(t) - 0.1 (\varphi_{i-1}(t) - 2\varphi_{i}(t) + \varphi_{i+1}(t)) + 0.4 \varphi_i(t) + 0.1 \varphi_i(t)^3 \right)}{ 0.0001\left(31 - t + \frac{1}{|i| + 1} \right)^2} \cdot \left(  0.6 + 0.3 \varphi_i(t)^2 \right) - 0.6 \varphi_i(t)\\
			&\quad - \frac{0.2 \left( \dot{\varphi}_{i-1}(t) - 0.1 (\varphi_{i-2}(t) - 2\varphi_{i-1}(t) + \varphi_{i}(t)) + 0.4 \varphi_{i-1}(t) + 0.1 \varphi_{i-1}(t)^3 \right)}{ 0.0001\left(31 - t + \frac{1}{|{i-1}| + 1} \right)^2}\\
			&\quad - \frac{0.2 \left( \dot{\varphi}_{i+1}(t) - 0.1 (\varphi_{i}(t) - 2\varphi_{i+1}(t) + \varphi_{i+2}(t)) + 0.4 \varphi_{i+1}(t) + 0.1 \varphi_{i+1}(t)^3 \right)}{ 0.0001\left(31 - t + \frac{1}{|{i+1}| + 1} \right)^2}.
		\end{aligned}
	\end{displaymath}

	Finally, we substitute these derivatives into the Euler-Lagrange equation:
	\begin{displaymath}
		\frac{d}{dt} \left( \frac{\partial OM}{\partial \dot{\varphi}_i(t)} \right) = \frac{\partial OM}{\partial \varphi_i(t)}.
	\end{displaymath}
	Expanding, we get
	\begin{displaymath}
		\begin{aligned}
			\ddot{\varphi}_i(t) &= \left(  \dot{\varphi}_i(t) - 0.1 (\varphi_{i-1}(t) - 2\varphi_{i}(t) + \varphi_{i+1}(t)) + 0.4 \varphi_i(t) + 0.1 \varphi_i(t)^3  \right)  \cdot \left( 0.6 + 0.3 \varphi_i(t)^2 \right)\\
			&\quad - {0.1 \left( \dot{\varphi}_{i-1}(t) - 0.1 (\varphi_{i-2}(t) - 2\varphi_{i-1}(t) + \varphi_{i}(t)) + 0.4 \varphi_{i-1}(t) + 0.1 \varphi_{i-1}(t)^3 \right)}\\
			&\quad - {0.1 \left( \dot{\varphi}_{i+1}(t) - 0.1 (\varphi_{i}(t) - 2\varphi_{i+1}(t) + \varphi_{i+2}(t)) + 0.4 \varphi_{i+1}(t) + 0.1 \varphi_{i+1}(t)^3 \right)}\\
			&\quad + \left(  0.1 (\dot{\varphi}_{i-1}(t) - 2\dot{\varphi}_{i}(t) + \dot{\varphi}_{i+1}(t)) - 0.4 \dot{\varphi}_i(t) - 0.3 \varphi_i(t)^2 \dot{\varphi}_i(t)\right) - 0.00003 \varphi_i(t){ \left(31 - t + \frac{1}{|i| + 1} \right)}^2\\
			&\quad -  \frac{2 \left( \dot{\varphi}_i(t) - 0.1 (\varphi_{i-1}(t) - 2\varphi_{i}(t) + \varphi_{i+1}(t)) + 0.4 \varphi_i(t) + 0.1 \varphi_i(t)^3  \right)}{\left( 31 - t + \frac{1}{|i| + 1} \right)},
		\end{aligned}
	\end{displaymath}
	with boundary conditions $\varphi_i(0) = 0.6\exp\left( - \frac{(i)^2}{2\sigma^2} \right)$ and $\varphi_i(30) = 0$. Here we complete the derivation of the Euler-Lagrange equations for the Onsager-Machlup functional. Subsequently, we use numerical methods to solve these equations. It is evident that $\eqref{40}$ satisfies Theorem $\ref{T3.3}$, enabling us to approximate the infinite-dimensional equations with finite-dimensional counterparts. We set $ i \in [-30, 30] $ and $ t \in [0, 30] $. $\mathrm{Fig.} \ref{F.1}$ presents the solution to $\eqref{40}$. $\mathrm{Fig.} \ref{F.2}$ illustrates the solution to the Euler-Lagrange equations of the Onsager-Machlup functional. $\mathrm{Fig.} \ref{F.3}$ shows two-dimensional cross-sectional plots of the solutions at $ i = 0 $ and $ i = 10 $, respectively.
	
	The graphical representations aid in visualizing the differences and similarities between the solutions to $\eqref{40}$ and the Euler-Lagrange equations. By comparing the numerical solutions, we can observe how the Onsager-Machlup functional accurately captures the dynamics of the stochastic system. This deeper understanding of the Onsager-Machlup functional in the context of stochastic lattice equations is crucial. The results validate the effectiveness of our derived Onsager-Machlup functional, which determines the most probable transition path connecting two given states among all possible smooth paths. This validation supports the theoretical framework and demonstrates the practical utility of the Onsager-Machlup functional in solving complex problems in stochastic lattice equations.
	
\begin{figure}[htbp]
	\centering
	\begin{minipage}[htbp]{8cm}
		\centering
		\includegraphics[width=8cm, trim=160pt 20pt 160pt 20pt, clip]{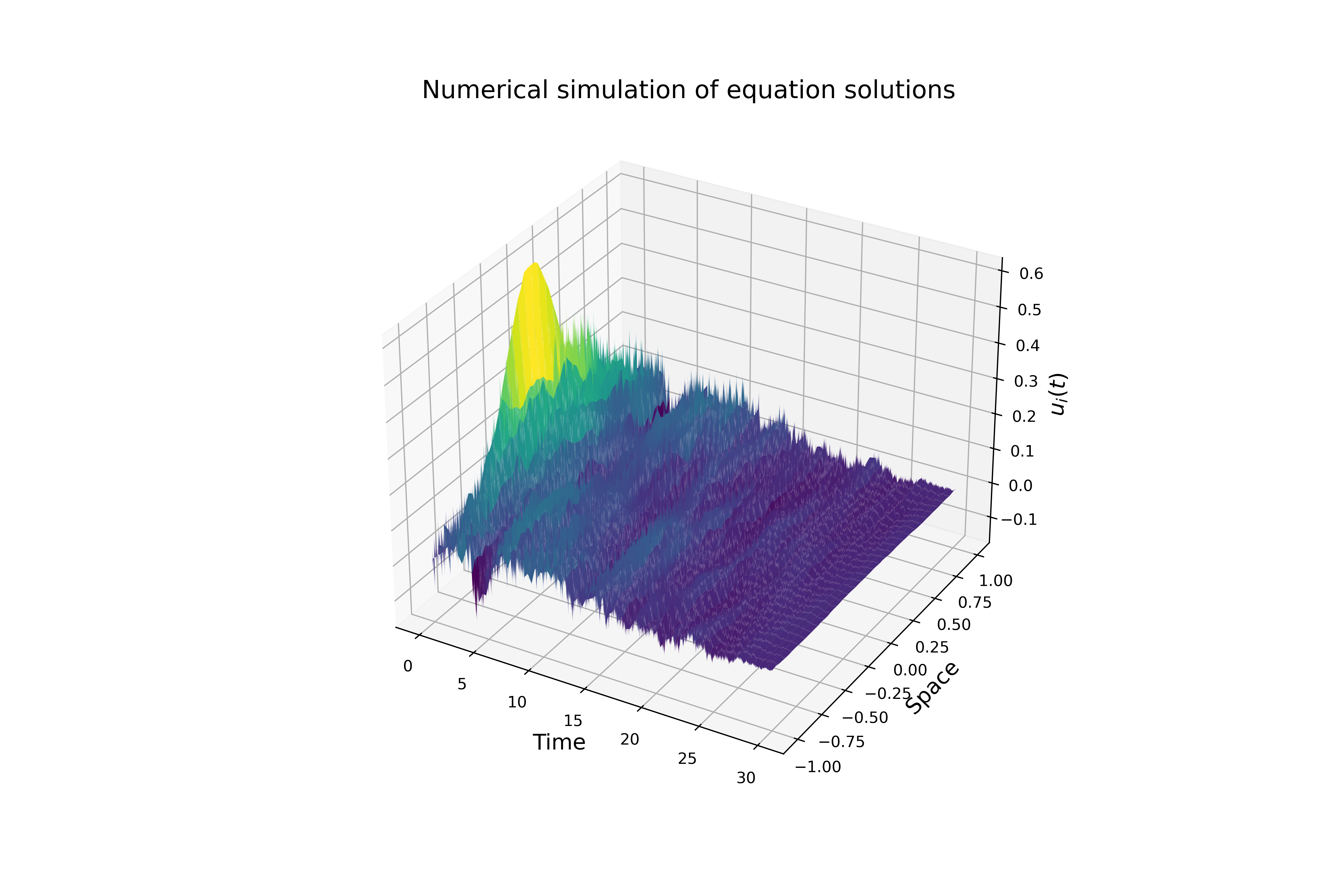}
		\caption{The graph of the solution $u(t)$ to equation $\eqref{40}$ on $t \in [0,30], i \in [-30,30]$.}
		\label{F.1}
	\end{minipage}
	\hspace{0.02\textwidth}
	\begin{minipage}[htbp]{8cm}
		\centering
		\includegraphics[width=8cm, trim=160pt 20pt 160pt 20pt, clip]{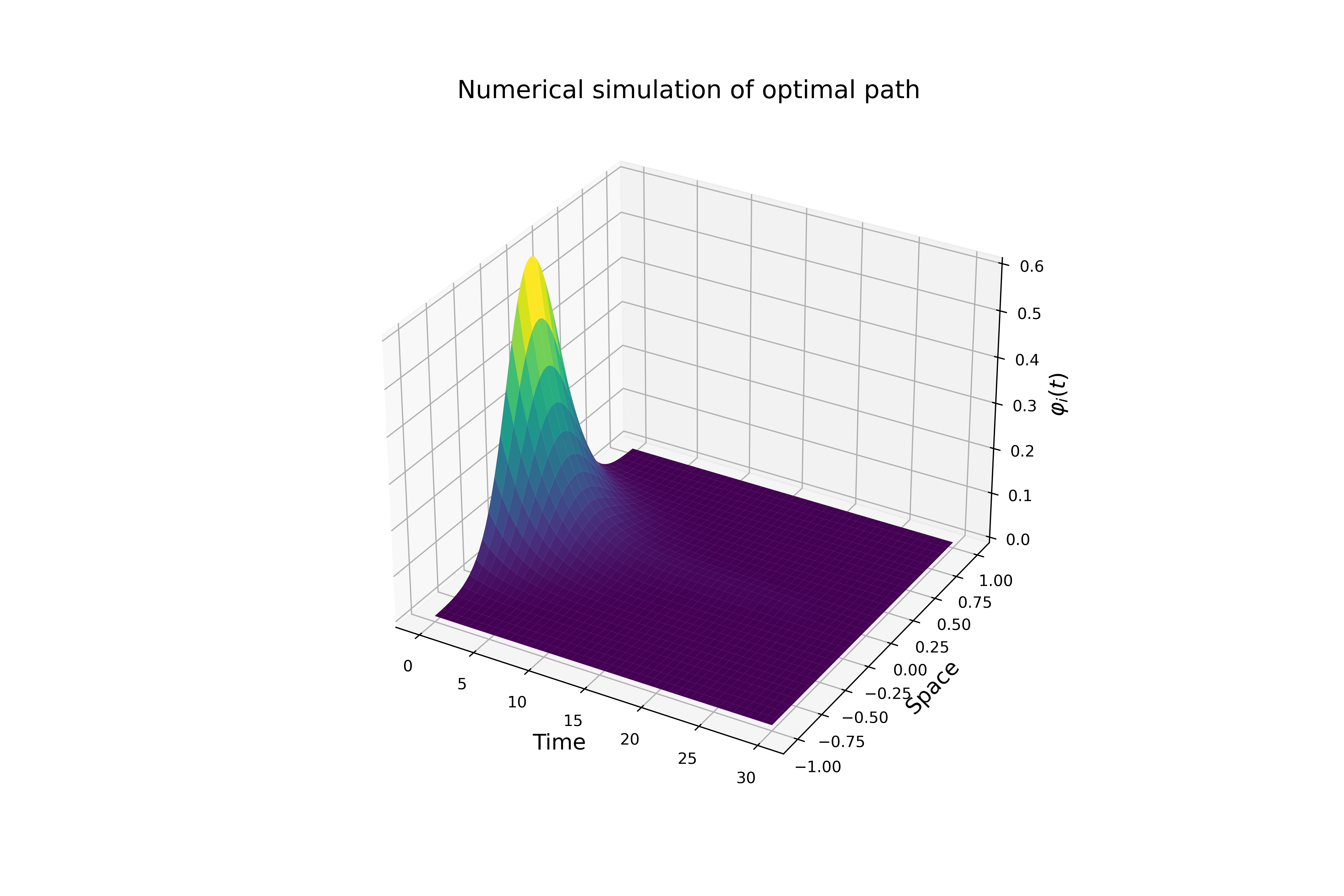}
		\caption{The graph of the optimal path $\varphi (t)$ on  $t \in [0,30], i \in [-30,30]$.}
		\label{F.2}
	\end{minipage}
\end{figure} 
	
	\begin{figure}[htbp]
		\centering
		\includegraphics[width=13cm]{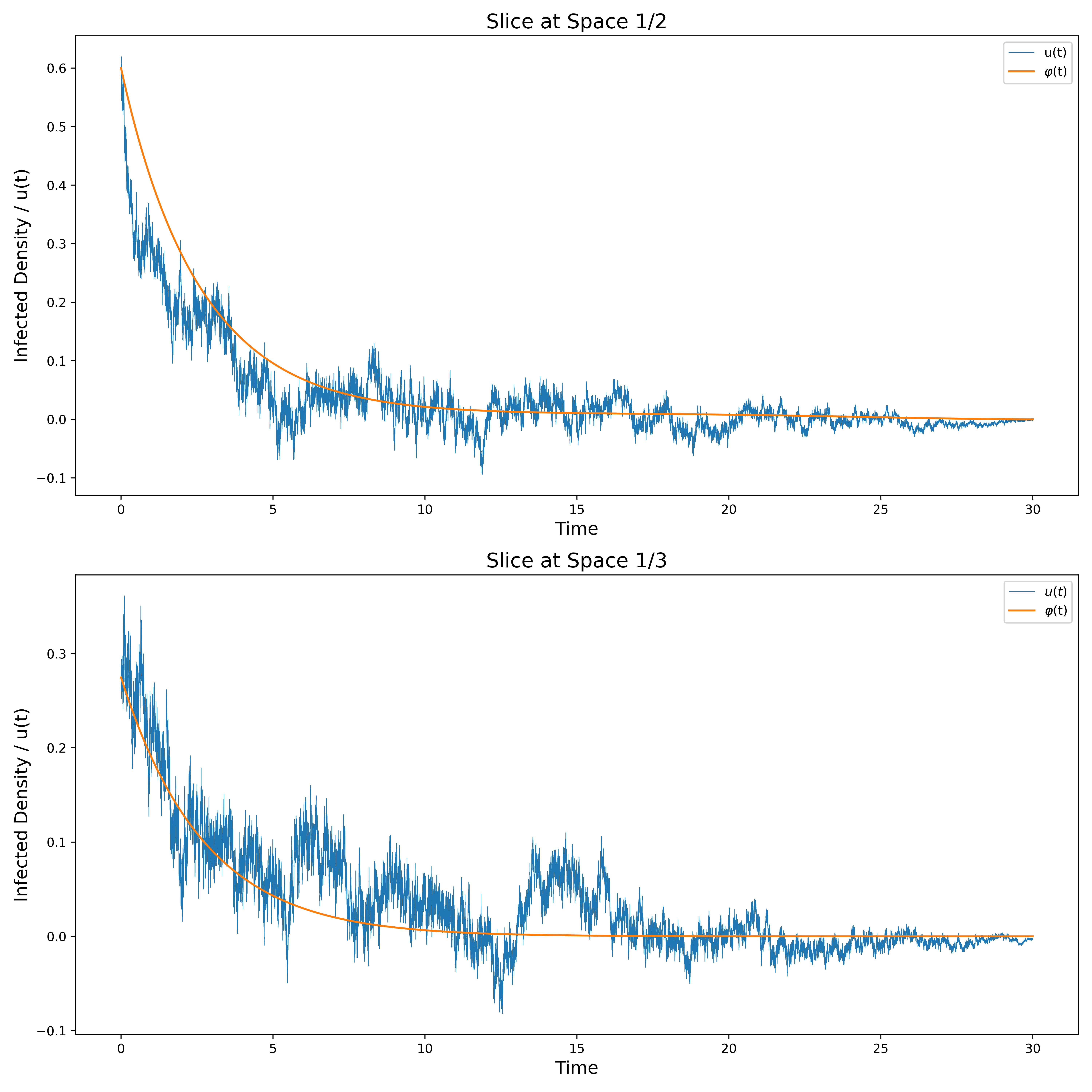}
		\caption{The slice of $\mathrm{Fig.} \ref{F.1}$ and $\mathrm{Fig.} \ref{F.2}$ at $ i = 0 $ and $ i = 10 $, respectively.}
		\label{F.3}
	\end{figure} 
\end{example}

\begin{funding}
	The second author (Y. Li) was supported by National Natural Science Foundation of China (Grant No. 12071175, 11171132, 11571065) and Natural Science Foundation of Jilin Province (20200201253JC).
\end{funding}

\bibliographystyle{plain}
\bibliography{ref.bib}

\end{document}